\documentclass[10pt]{amsart}
\usepackage{amssymb,mathrsfs,skak}

\def\bR {\mathbf{R}}

\def\cA {\mathcal{A}}
\def\cB {\mathcal{B}}

\def\cD {\mathcal{D}}

\def\cH {\mathcal{H}}

\def\cM {\mathcal{M}}
\def\cN {\mathcal{N}}
\def\cP {\mathcal{P}}
\def\cQ {\mathcal{Q}}
\def\cR {\mathcal{R}}
\def\cS {\mathcal{S}}

\def\cV {\mathcal{V}}

\def\scrL{\mathscr{L}}

\def\b {{\beta}}

\def\de {{\delta}}
\def\eps {{\epsilon}}
\def\th {{\theta}}
\def\vth{{\vartheta}}
\def\Th {{\Theta}}
\def\ka {{\kappa}}

\def\si {{\sigma}}

\def\Om {{\Omega}}

\def\d {{\partial}}
\def\grad {{\nabla}}
\def\Dlt {{\Delta}}

\def\rstr {{\big |}}
\def\Rstr {{\Bigg |}}
\def\indc {{\bf 1}}

\def\la {\langle}
\def\ra {\rangle}
\def \La {\bigg\langle}
\def \Ra {\bigg\rangle}

\newcommand{\Div}{\operatorname{div}}

\newcommand{\Supp}{\operatorname{supp}}

\newcommand{\ba}{\begin{aligned}}
\newcommand{\ea}{\end{aligned}}

\newcommand{\be}{\begin{equation}}
\newcommand{\ee}{\end{equation}}

\newcommand{\lb}{\label}

\newtheorem{Thm}{Theorem}[section]

\newtheorem{Prop}[Thm]{Proposition}

\newtheorem{Lem}[Thm]{Lemma}

\begin{document}

\title[Homogenized two-temperature model]{Derivation of a homogenized two-temperature model from the heat equation}

\author[L. Desvillettes]{Laurent Desvillettes}
\address[L.D.]{Ecole Normale Sup\'erieure de Cachan, CMLA, 61, Av. du Pdt. Wilson, 94235 Cachan Cedex, France}
\email{desville@cmla.ens-cachan.fr}

\author[F. Golse]{Fran\c cois Golse}
\address[F.G.]{Ecole Polytechnique, Centre de Math\'ematiques L. Schwartz, 91128 Palaiseau Cedex, France}
\email{francois.golse@math.polytechnique.fr}

\author[V. Ricci]{Valeria Ricci}
\address[V.R.]{ Dipartimento di Metodi e Modelli Matematici, Universit\`a di Palermo, Viale delle Scienze Edificio 8, 90128 Palermo, Italy}
\email{ricci@unipa.it}

\begin{abstract}
This work studies the heat equation in a two-phase material with spherical inclusions. Under some appropriate scaling on the size, volume fraction 
and heat capacity of the inclusions, we derive a coupled system of partial differential equations governing the evolution of the temperature of 
each phase at a macroscopic level of description. The coupling terms describing the exchange of heat between the phases are obtained by using 
homogenization techniques originating from [D. Cioranescu, F. Murat: Coll\`ege de France Seminar vol. 2. (Paris 1979-1980) Res. Notes in Math. 
vol. 60, pp. 98--138. Pitman, Boston, London, 1982.]
\end{abstract}

\keywords{Heat equation; Homogenization; Infinite diffusion limit; Thermal nonequilibrium models}

\subjclass{35K05, 35B27,76T05, (35Q79, 76M50)}

\maketitle


\section{Description of the problem}


\subsection{The homogenized two-temperature model}\lb{subsec_1.1}

The purpose of this paper is to derive a model governing the exchange of heat in a composite medium consisting of a background material with very
small spherical inclusions of another material with large thermal conductivity. Specifically, we assume that the volume fraction of the inclusions is 
negligible, while the heat capacity of each inclusion is large. 

Under some appropriate scaling assumptions on the size, volume fraction and heat capacity of the inclusions, the temperature field $T\equiv T(t,x)$ of 
the background material and the temperature field $\th\equiv\th(t,x)$ of the dispersed phase (i.e. the inclusions) satisfy
\begin{equation}\label{efina}
\left\{
\ba
{}&\d_tT-d\Dlt_xT+4\pi\rho d(T-\th)=0\,,
\\
&\frac{d}{d'}\d_t\th+4\pi\rho d(\th-T)=0\,,
\ea
\right.
\end{equation}
where $\rho\equiv\rho(x)$ is the number density of inclusions while $d$ and $d'$ are the heat diffusion coefficients (i.e. the ratio of the heat conductivity
to the volumetric heat capacity) of the background material and the inclusions respectively. 

Our work is motivated by a class of models used in the theory of multiphase flows, especially of mutiphase flows in porous media. In such flows, each 
phase can have its own temperature (in which case the flow is said to be in thermal local non-equilibrium). Averaged equations for those temperatures 
similar to (\ref{efina}) have been proposed in \cite{PFQ1999}, \cite{FDTBQ2006} and \cite{Brennen2005} on the basis of arguments at a macroscopic
level of description. While these references address the case of complex realistic flows, our setting is purposedly chosen as simple as possible. Neither 
convection nor phase changes are taken into account in our model. Besides we only consider two phase flows, with only one phase having a positive
diffusion rate. The case of positive diffusion rates is considered in \cite{PFQ1999} (eq. (13) -- (15) on p. 242) and in \cite{FDTBQ2006} (on p. 2151),
while the case of phases with zero diffusion rate is considered in \cite{Brennen2005} (eqs. (1.63) and (1.74) on p. 39).

\medskip
We give a rigorous derivation of the coupled system above from a model where the heat conductivity of the dispersed phase is assumed to be infinite 
from the outset. Our derivation is based on homogenization arguments following our earlier work in \cite{DGR1}, inspired from \cite{Hruslov,CioraMura}. 

For the sake of being complete, we also give a rigorous derivation of the infinite conductivity model from the classical heat diffusion equation. In the next 
two sections, we briefly describe the heat diffusion problem in a binary composite material, and the infinite conductivity model that is our starting point
for the homogenization process.

\subsection{The model with finite conductivity} \label{subsec_1.2}

Consider an open domain $\Om\subset\bR^3$, let $A$ be an open subset of $\Om$ and let $B=\Om\setminus A$ be closed in $\bR^3$. Assume that 
$\d\Om$ and $\d B$ are submanifolds of $\bR^3$ of class $C^2$, and that $B\cap\d\Om=\varnothing$. Notice that our results would also hold in the 
case $\Om=\bR^3$. The unit normal field on the boundary of $B$ is oriented towards $A$.

The set $A$ is occupied by a material $\cA$ with heat conductivity $\ka_A$, density $\rho_A$ and specific heat capacity $C_A$, while the set $B$ is 
occupied by a material $\cB$ with heat conductivity $\ka_B$, density $\rho_B$ and specific heat capacity $C_B$. It will be assumed that 
$\rho_A,C_A,\ka_A,\rho_B,C_B,\ka_B$ are continuous positive functions on $\overline A$ and $B$ respectively. Denote by $T_A: = T_A(t,x)>0$  and 
$T_B:=T_B(t,x)>0$ the  temperatures of $\cA$ and $\cB$ at time $t>0$ and point $x \in A$ ($x\in B$ respectively).

Assuming that Fourier's law holds in both materials and that $T_A$ and $T_B$ are smooth (at least of class $C^2$) one has
\be\lb{HeatEqA+B}
\ba
\rho_A(x)C_A(x)\d_tT_A(t,x)=\,\Div_x(\ka_A(x)\grad_xT_A(t,x))\,,\qquad x\in A\,,\,\,t>0\,,
\\
\rho_B(x)C_B(x)\d_tT_B(t,x)=\Div_x(\ka_B(x)\grad_xT_B(t,x))\,,\qquad x\in \mathring{B}\,,\,\,t>0\,.
\ea
\ee

If there is no heat source concentrated on the interface $\d B$, then the temperature varies continuously across the interface between material $\cA$ and 
material $\cB$ and there is no net heat flux across that same interface. In other words, assuming that $T_A$ and $T_B$ are smooth up to the interface
$\d B$ between both materials
\be\lb{InterfCond}
\left\{
\ba
{}&T_A(t,x)=T_B(t,x)\,,&&\qquad x\in\d B\,,\,\,t>0\,,
\\
&\ka_A(x)\frac{\d T_A}{\d n }(t,x)=\ka_B(x)\frac{\d T_B}{\d n}(t,x)\,,&&\qquad x\in\d B\,,\,\,t>0\,.
\ea
\right.
\ee

Define
\be\lb{DefRhoC}
\rho(x):=\left\{\begin{array}{l}\rho_A(x)\quad x\in A\\ \rho_B(x)\quad x\in B\end{array}\right.\quad
C(x):=\left\{\begin{array}{l}C_A(x)\quad x\in A\\ C_B(x)\quad x\in B\end{array}\right.
\ee
together with
\be\lb{DefKa}
\ka(x):=\left\{\begin{array}{l}\ka_A(x)\quad x\in A\\ \ka_B(x)\quad x\in B\end{array}\right.
\ee
and
\be\lb{DefT}
T(t,x):=\left\{\begin{array}{l}T_A(t,x)\quad x\in A\\ T_B(t,x)\quad x\in B\end{array}\right.
\ee

Assume that 
\be\lb{HypTATB}
\left\{
\ba
T_A\in C([0,\tau];L^2(A))\cap L^2(0,\tau;H^1(A))\,,
\\
T_B\in C([0,\tau];L^2(B))\cap L^2(0,\tau;H^1(B))\,.
\ea
\right.
\ee
In that case the functions $T_A$ and $T_B$ have traces on $\d B$ denoted $T_A\rstr_{\d B}$ and $T_B\rstr_{\d B}$ belonging to $L^2(0,\tau;H^{1/2}(\d B))$. 

Moreover, if $T_A$ and $T_B$ satisfy (\ref{HeatEqA+B}), the vector fields
$$
(\rho_AC_AT_A,-\ka_A\grad_xT_A)\hbox{ and }(\rho_BC_BT_B,-\ka_B\grad_xT_B)
$$
are divergence free in $(0,\tau)\times A$ and $(0,\tau)\times\mathring{B}$ respectively. By statement a) in Lemma \ref{L-NormTr}, both sides of the 
second equality in (\ref{InterfCond}) are well defined elements of $H^{1/2}_{00}((0,\tau)\times\d B)'$. (We recall that $H^{1/2}_{00}((0,\tau)\times\d B)$ 
is the Lions-Magenes subspace of functions in $H^{1/2}((0,\tau)\times\d B)$ whose extension by $0$ to $\bR\times\d B$ defines an element of 
$H^{1/2}(\bR\times\d B)$; the notation $H^{1/2}_{00}((0,\tau)\times\d B)'$ designates the dual of that space.)

\begin{Lem}\lb{L-EquivHeatEqA+B-Om}
Assume that $T_A$ and $T_B$ satisfy assumptions (\ref{HypTATB}). Let $\rho,C,\ka$ and $T$ be defined as in (\ref{DefRhoC})-(\ref{DefKa}) and (\ref{DefT}).
Then
$$
T\in C([0,\tau];L^2(\Om))\cap L^2(0,\tau;H^1(\Om))
$$
and   
\be\lb{HeatEqOm}
\rho(x)C(x)\d_tT(t,x)=\Div_x(\ka(x)\grad_xT(t,x))\,\quad x\in\Om\,,\,\,t>0
\ee
holds in the sense of distributions in $(0,\tau)\times\Om$ if and only if both (\ref{HeatEqA+B}) and (\ref{InterfCond}) hold in the sense of distributions.
\end{Lem}

\begin{proof}
Under the assumption (\ref{HypTATB}), the function $T$ defined by (\ref{DefT}) belongs to the space $L^2((0,\tau);H^1(\Om))$ if and only if the boundary traces of 
$T_A$ and $T_B$ coincide, i.e.
$$
T_A(t,\cdot)\rstr_{\d B}=T_B(t,\cdot)\rstr_{\d B}\qquad\hbox{ for a.e. }t\in[0,\tau]\,.
$$

If (\ref{HeatEqOm}) holds in the sense of distributions on $(0,\tau)\times\Om$, then (\ref{HeatEqA+B}) hold in the sense of distributions on $(0,\tau)\times A$
and $(0,\tau)\times B$ respectively. 

For $\phi\in C^\infty_c(\Om)$, one has
$$
\ba
\frac{d}{dt}\int_\Om\rho(x)C(x)T(t,x)\phi(x)dx+\int_\Om\ka(x)\grad_xT(t,x)\cdot\grad\phi(x)dx
\\
=
\frac{d}{dt}\int_A\rho_A(x)C_A(x)T_A(t,x)\phi(x)dx+\int_A\ka_A(x)\grad_xT_A(t,x)\cdot\grad\phi(x)dx
\\
+\frac{d}{dt}\int_B\rho_B(x)C_B(x)T_B(t,x)\phi(x)dx+\int_B\ka_B(x)\grad_xT_B(t,x)\cdot\grad\phi(x)dx
\\
=-\La\ka_A\frac{\d T_A}{\d n},\phi\Ra_{H^{-1/2}(\d B),H^{1/2}(\d B)}+\La\ka_B\frac{\d T_B}{\d n},\phi\Ra_{H^{-1/2}(\d B),H^{1/2}(\d B)}
\ea
$$
provided that $T_A$ and $T_B$ satisfy (\ref{HeatEqA+B}), by statement b) in Lemma \ref{L-NormTr}. 

Thus, if $T$ satisfies (\ref{HeatEqOm}) in the sense of distributions on $(0,\tau)\times\Om$, then $T_A$ and $T_B$ satisfy (\ref{HeatEqA+B}) on
$(0,\tau)\times A$ and $(0,\tau)\times\mathring{B}$ respectively. Therefore the identity above holds with left hand side equal to $0$ in the sense of
distributions on $(0,\tau)$, so that 
$$
\La\ka_A\frac{\d T_A}{\d n}-\ka_B\frac{\d T_B}{\d n},\phi\Ra_{H^{-1/2}(\d B),H^{1/2}(\d B)}=0\quad\hbox{ in }\cD'((0,\tau))\,.
$$
This implies in turn the second equality in (\ref{InterfCond}).

Conversely, if $T_A$ and $T_B$ satisfy (\ref{HeatEqA+B}), the above identity holds with right hand side equal to $0$ by the second equality in
(\ref{InterfCond}). Therefore
$$
\frac{d}{dt}\int_\Om\rho(x)C(x)T(t,x)\phi(x)dx+\int_\Om\ka(x)\grad_xT(t,x)\cdot\grad\phi(x)dx=0
$$
for all $\phi\in C^\infty_c(\Om)$, which implies that (\ref{HeatEqOm}) holds in the sense of distributions on $(0,\tau)\times\Om$ by a classical density
argument.
\end{proof}

\smallskip
Therefore, we start from the heat equation (\ref{HeatEqOm}) with $\rho,C,\ka$ as in (\ref{DefRhoC}), (\ref{DefKa}) and we assume that there is no 
heat flux across $\d\Om$, in other words that $T$ satisfies the Neumann boundary condition
\be\lb{NeumOm}
\frac{\d T}{\d n}(t,x)=0\,,\quad x\in\d\Om\,,\,\,t>0\,.
\ee

\subsection{The model with infinite conductivity} \lb{subsec_1.3}

In this section we assume that $B$ has $N$ connected components denoted $B_i$ for $i=1,\ldots,N$. 

Our first task is to derive the governing equation for the temperature field $T$ in $\Om$ when the material $\cB$ filling $B$ has infinite heat conductivity. 
In that case the temperature $T$ instantaneously reaches equilibrium in each connected component $B_i$ of $B$, so that
\be\lb{DefTi}
T(t,x)=T_i(t)\,,\quad x\in B_i\,,\,\,t>0
\ee
for each $i=1,\ldots,N$. Therefore, the unknown for the problem with infinite conductivity is $(T_A(t,x),T_1(t),\ldots,T_N(t))$, where
\be\lb{HeatEqInfA}
\left\{
\ba
{}&\rho_A(x)C_A(x)\d_tT_A(t,x)=\Div_x(\ka_A(x)\grad_xT_A(t,x))\,,\quad &&x\in A\,,\,\,t>0\,,
\\
&\frac{\d T}{\d n}(t,x)=0\,, &&x\in\d\Om\,,\,\,t>0
\\
&T_A(t,x)=T_i(t)\,, &&x\in\d B_i\,,\,\,t>0
\ea
\right.
\ee
This is obviously not enough to determine the evolution of $T_A$ and of $T_i$ for all $i=1,\ldots,N$. 

For finite $\ka_B$, the vector field
$$
(t,x)\mapsto(\rho_B(x)C_B(x)T_B(t,x),-\ka_B(x)\grad_xT_B(t,x))
$$
is divergence free in $(0,\tau)\times B_i$ for each $i=1,\ldots,N$. By statement b) in Lemma \ref{L-NormTr} and the second equality in (\ref{InterfCond})
$$
\ba
\frac{d}{dt}\int_{B_i}\rho_B(x)C_B(x)T_B(t,x)dx&=\La\ka_B(x)\frac{\d T_B}{\d n}(t,\cdot),1\Ra_{H^{-1/2}(\d B_i),H^{1/2}(\d B_i)}
\\
&=\La\ka_A(x)\frac{\d T_A}{\d n}(t,\cdot),1\Ra_{H^{-1/2}(\d B_i),H^{1/2}(\d B_i)}\,.
\ea
$$
Letting $\ka_B\to\infty$ and abusing the integral notation to designate the last duality bracket above, one uses (\ref{DefTi}) to conclude that
\be\lb{HeatEqInfB}
\dot{T}_i(t)=\frac1{\b_i}\int_{\d B_i}\ka_A(x)\frac{\d T_A}{\d n}(t,x)dS(x)
\ee
where
\be\lb{DefBeti}
\b_i:=\int_{B_i}\rho_B(x)C_B(x)dx\,.
\ee

The argument above suggests that the governing equations for the infinite conductivity problem with unknowns $(T_A(t,x),T_1(t),\ldots,T_N(t))$ is the 
system consisting of (\ref{HeatEqInfA}) with (\ref{HeatEqInfB}) for $i=1,\ldots,N$.


\section{Main results}


\subsection{Existence and uniqueness theory for the heat equation with discontinuous coefficients}

Since our starting point is (\ref{HeatEqA+B}) with interface condition (\ref{InterfCond}), or equivalently the heat equation (\ref{HeatEqOm}) with 
discontinuous coefficients (see Lemma \ref{L-EquivHeatEqA+B-Om}), we first recall the existence and uniqueness theory for (\ref{HeatEqOm}) 
with Neumann boundary condition (\ref{NeumOm}). Except for the possibly non smooth factor $\rho(x)C(x)$, this is a classical result. This 
factor can be handled with appropriate weighted Sobolev spaces; for the sake of being complete, we sketch the (elementary) argument below.

\begin{Prop}\lb{P-FinCond}
Let $\ka\equiv\ka(x)$, $\rho\equiv\rho(x)$ and $C\equiv C(x)$ be measurable functions on $\Om$ satisfying
$$
\ka_m\le\ka(x)\le\ka_M\,,\quad\rho_m\le\rho(x)\le\rho_M\,,\quad C_m\le C(x)\le C_M
$$
for a.e. $x\in\Om$, where $\ka_m,\ka_M,\rho_m,\rho_M,C_m,C_M>0$, and let $T^{in}\in L^2(\Om)$. There exists a unique 
$$
T\in C_b([0,+\infty);L^2(\Om))\cap L^2(0,\tau;H^1(\Om))
$$
for each $\tau>0$ that is a weak solution of the problem
\be\lb{HeatNeum}
\left\{
\ba
{}&\rho(x)C(x)\d_tT(t,x)=\Div_x(\ka(x)\grad_xT(t,x))\,,\quad &&x\in\Om\,,\,\,t>0\,,
\\
&\frac{\d T}{\d n}(t,x)=0\,,&&x\in\d\Om\,,\,\,t>0\,,
\\
&T(0,x)=T^{in}(x)\,,&&x\in\Om\,.
\ea
\right.
\ee
This solution satisfies
$$
\rho C\d_tT\in L^2(0,\tau;H^1(\Om)')\,,
$$
for each $\tau>0$, together with the ``energy'' identity
$$
\tfrac12\int_\Om\rho(x)C(x)T(t,x)^2dx+\int_0^t\int_\Om\ka(x)|\grad_xT(t,x)|^2dxdt=\tfrac12\int_\Om\rho(x)C(x)T^{in}(x)^2dx
$$
for each $t>0$.
\end{Prop}

We recall the weak formulation of (\ref{HeatNeum}): for each $w\in H^1(\Om)$
$$
\la\rho C\d_tT(t,\cdot),w\ra_{H^1(\Om)',H^1(\Om)}+\int_\Om\ka(x)\grad_xT(t,x)\cdot\grad_xw(t,x)dx=0\hbox{ for a.e. }t\ge 0\,.
$$
The Neumann condition in (\ref{HeatNeum}) is contained in the choice of $L^2([0,+\infty);H^1(\Om))$ as the set of test functions in the weak formulation
above, while there is no difficulty with initial condition since $T\in C([0,+\infty);L^2(\Om))$.

\newpage
\subsection{The infinite conductivity limit}

\subsubsection{Variational formulation of the infinite conductivity problem}

Assume as in section \ref{subsec_1.3} that $B$ has $N$ connected components denoted $B_i$ for $i=1,\ldots,N$. The heat diffusion problem with 
infinite heat conductivity in $B$ is:
\be\lb{InfHeatCond}
\left\{
\ba
{}&\rho_AC_A\d_tT(t,x)=\Div_x(\ka_A\grad_xT(t,x))\,,\quad &&x\in A\,,\,\,t>0\,,
\\
&\frac{\d T}{\d n}(t,x)=0\,,&&x\in\d\Om\,,\,\,t>0\,,
\\
&T(t,x)=T_i(t)\,,&&x\in\d B_i\,,\,\,t>0\,,\,\,1\le i\le N\,,
\\
&\b_i\dot{T}_i(t)=\int_{\d B_i}\ka_A\frac{\d T}{\d n}(t,x)dS(x)\,,&&t>0\,,\,\,1\le i\le N\,,
\\
&T(0,x)=T^{in}(x)\,,&&x\in\Om\,.
\ea
\right.
\ee
Its variational formulation is as follows. 

Let $\cH_N$ be the closed subspace of $L^2(\Om)$ defined as 
$$
\cH_N:=\left\{u\in L^2(\Om)\hbox{ s.t. }u(x)=\frac1{|B_i|}\int_{B_i}u(y)dy\hbox{ for a.e. }x\in B_i\,,\,\,i=1,\ldots,N\right\}\,;
$$
and equipped with the inner product 
$$
(u|v)_{\cH_N}=\int_\Om u(x)v(x)\rho(x)C(x)dx\,.
$$
Define
$$
\cV_N:=\cH_N\cap H^1(\Om)
$$
with the inner product
$$
(u|v)_{\cV_N}=(u|v)_{\cH_N}+\int_A\grad u(x)\cdot\grad v(x)\rho_A(x)C_A(x)dx\,.
$$
Obviously $\cV_N$ is a separable Hilbert space, the inclusion $\cV_N\subset\cH_N$ is continuous and $\cV_N$ is a dense subspace of $\cH_N$. 
Besides, the map $\cH_N\ni u\mapsto L_u\in\cV_N'$, where $L_u$ is the linear functional $v\mapsto(u|v)_{\cH_N}$, identifies $\cH_N$ with a dense
subspace of $\cV_N'$.                             

The variational formulation of the infinite conductivity problem is as follows: a weak solution of (\ref{InfHeatCond}) is a function 
\be\lb{CondTInfCond}
T\in C([0,\tau];\cH_N)\cap L^2(0,\tau;\cV_N)\hbox{ such that }\rho C\d_tT\in L^2(0,\tau;\cV_N')
\ee
satisfying the initial condition and
\be\lb{VarInfCond}
\left\{
\ba
&\d_t(T(t,\cdot)|w)_{\cH_N}+\int_A\ka_A(x)\grad_xT(t,x)\cdot\grad w(x)dx=0\hbox{ for a.e. }t\in[0,\tau]
\\
&\hbox{for each test function }w\in\cV_N
\ea
\right.
\ee
This variational formulation is justified by the following observation.

\begin{Prop}\lb{P-VarFormInf}
Let $T$ satisfy (\ref{CondTInfCond}) and the initial condition in (\ref{InfHeatCond}). 

If $T$ satisfies the variational condition (\ref{VarInfCond}), then
\be\lb{HeatEqA}
\rho_AC_A\d_tT=\Div_x(\ka_A\grad_xT)\quad\hbox{ in }\cD'((0,\tau)\times A)
\ee
and 
\be\lb{NeumCond2}
\ka_A\frac{\d T}{\d n}\Rstr_{(0,\tau)\times\d\Om}=0\quad\hbox{ in }H^{1/2}_{00}((0,\tau)\times\d\Om)'
\ee
while
\be\lb{TransmCondi}
\b_i\dot{T}_i=\La\ka_A\frac{\d T}{\d n}\Rstr_{\d B_i},1\Ra_{H^{-1/2}(\d B_i),H^{1/2}(\d B_i)}\hbox{ in }H^{-1}((0,\tau))
\ee
for each $i=1,\ldots,N$, where
$$
T_i(t):=\frac1{|B_i|}\int_{B_i}T(t,x)dx\,.
$$

Conversely, if $T$ satisfies both (\ref{HeatEqA}), (\ref{NeumCond2}) and (\ref{TransmCondi}), it must satisfy the variational formulation (\ref{VarInfCond}).
\end{Prop}

\smallskip
The existence and uniqueness of a weak solution of the infinite heat conductivity problem is given in the next proposition.

\begin{Prop}\lb{P-InfCond}
Assume that $\ka_A$ is a measurable function defined a.e. on $A$ satisfying
\be\lb{HypKaA}
\ka_m\le\ka_A(x)\le\ka_M\,,\qquad\hbox{ for a.e. }x\in A\,,
\ee
where $\ka_m$ and $\ka_M$ are positive numbers, while $\rho$ and $C$ satisfy the same assumptions as in Proposition \ref{P-FinCond}. Then for each 
$T^{in}\in\cH_N$, there exists a unique weak solution $T$ of (\ref{InfHeatCond}) defined for all $t\in[0,+\infty)$. This solution satisfies 
$$
\rho C\d_tT\in L^2(0,\tau;\cV_N')\,,
$$
for all $\tau>0$, together with the ``energy'' identity
$$
\ba
\tfrac12\int_A\rho_A(x)C_A(x)T(t,x)^2dx+\tfrac12\sum_{i=1}^N\b_iT_i(t)^2+\int_0^t\int_A\ka_A(x)|\grad_xT(t,x)|^2dxdt
\\
=
\tfrac12\int_A\rho_A(x)C_A(x)T^{in}(x)^2dx+\tfrac12\sum_{i=1}^N\b_i|T^{in}_i|^2
\ea
$$
for each $t>0$, where
$$
T_i(t):=\frac1{|B_i|}\int_{B_i}T(t,y)dy\quad\hbox{ and }\quad T^{in}_i:=\frac1{|B_i|}\int_{B_i}T^{in}(y)dy\,.
$$
\end{Prop}

\smallskip
Notice that this existence and uniqueness result assumes that the initial temperature field $T^{in}$ is a constant in each connected component of $B$.
This assumption is implied by the requirement that $T^{in}\in\cH_N$. While this restriction may seem questionable, it is very natural from the
mathematical viewpoint. For general initial temperature fields $T^{in}$, the solution of (\ref{InfHeatCond}) would include an initial layer corresponding 
with the relaxation to thermal equilibrium in each connected component of $B$. Such initial layers involve fast variations of the temperature field 
that are incompatible with the condition $\rho C\d_tT\in L^2(0,\tau;\cV_N')$ in the infinite conductivity limit.

\subsubsection{Convergence to the infinite conductivity model}

For each $\eta>0$, let $\ka_\eta$ be defined as follows:
\be\lb{Kaeta}
\ka_\eta(x):=\left\{\ba{}&\ka_A(x)\quad &&x\in A\\&\ka_B(x)/\eta&&x\in B\ea\right.
\ee
where $\ka_A$ and $\ka_B$ are measurable functions on $A$ and $B$ respectively satisfying
\be\lb{HypKaAB}
\ka_m\le\ka_A(x)\le\ka_M\hbox{ and }\ka_m\le\ka_B(y)\le\ka_M\,,\qquad\hbox{ for a.e. }x\in A\,\hbox{Êand }y\in B\,,
\ee
$\ka_M$ and $\ka_m$ being two positive constants.

\begin{Thm}\lb{T-InfCondLim}
Assume that $\rho$ and $C$ satisfy the same assumptions as in Proposition \ref{P-FinCond}, while $\ka_A$ and $\ka_B$ satisfy (\ref{HypKaAB}).
Let $T^{in}\in\cH_N$. For each $\eta>0$, let $T_\eta\in C_b([0,+\infty);L^2(\Om))\cap L^2(0,\tau;H^1(\Om))$ for all $\tau>0$ be the weak solution of 
(\ref{HeatNeum}) with heat conductivity $\ka_\eta$ defined as in (\ref{Kaeta}) and initial data $T^{in}$. Then
$$
T_\eta\to T\hbox{ in }L^2(0,\tau;H^1(\Om))
$$
as $\eta\to 0$ for all $\tau>0$, where $T\in C_b([0,+\infty);\cH_N)\cap L^2(0,\tau;\cV_N)$ for all $\tau>0$ is the weak solution of the infinite conductivity 
problem (\ref{InfHeatCond}).
\end{Thm}

\subsection{The homogenized system}

Let $\si,\si'>0$ and let $\rho\in C_b(\overline\Om)$ be a probability density on $\overline\Om$ such that $1/\rho$ is bounded on $\overline\Om$. Let 
$T^{in},\vth^{in}\in L^2(\Om)$. Consider the system
\be\lb{HomSyst}
\left\{
\ba
{}&(\d_t-\si\Dlt_x)T(t,x)+4\pi\si(\rho(x)T(t,x)-\vth(t,x))=0\,,&&\quad x\in\Om\,,\,\,t>0\,,
\\
&\d_t\vth(t,x)+4\pi\si'(\vth(t,x)-\rho(x)T(t,x))=0\,,&&\quad x\in\Om\,,\,\,t>0\,,
\\
&\frac{\d T}{\d n}(t,x)=0\,,&&\quad x\in\d\Om\,,\,\,t>0\,,
\\
&T(0,x)=T^{in}(x)\,,\quad\vth(0,x)=\vth^{in}(x)\,,&&\quad x\in\Om\,.
\ea
\right.
\ee

A weak solution of (\ref{HomSyst}) is a pair $(T,\vth)$ such that
$$
T\in L^\infty([0,+\infty);L^2(\Om))\cap L^2(0,\tau;H^1(\Om))\quad\hbox{ and }\vth\in L^\infty([0,+\infty);L^2(\Om))\,,
$$
for all $\tau>0$, and
$$
\ba
\frac{d}{dt}\int_\Om T(t,x)\phi(x)dx&+\si\int_\Om\grad_xT(t,x)\cdot\grad\phi(x)dx
\\
&+4\pi\si\int_\Om(\rho(x)T(t,x)-\vth(t,x))\phi(x)dx=0
\\
\frac{d}{dt}\int_\Om\vth(t,x)\psi(x)dx&+4\pi\si'\int_\Om(\vth(t,x)-\rho(x)T(t,x))\psi(x)dx=0
\ea
$$
in the sense of distributions on $(0,+\infty)$ for each $\phi\in H^1(\Om)$ and $\psi\in L^2(\Om)$, together with the initial condition. Observe that
the identities above imply that
$$
\frac{d}{dt}\int_\Om T(t,x)\phi(x)dx\quad\hbox{ and }\frac{d}{dt}\int_\Om\vth(t,x)\psi(x)dx\in L^2([0,\tau])
$$
for each $\tau>0$, so that the functions
$$
t\mapsto\int_\Om T(t,x)\phi(x)dx\quad\hbox{ and }t\mapsto\int_\Om\vth(t,x)\psi(x)dx
$$
are continuous on $[0,+\infty)$. Therefore the initial condition, interpreted as
$$
\int_\Om T(0,x)\phi(x)dx=\int_\Om T^{in}(x)\phi(x)dx\,,\quad\int_\Om\vth(0,x)\psi(x)dx=\int_\Om\vth^{in}(x)\psi(x)dx
$$
for all $\phi\in H^1(\Om)$ and all $\psi\in L^2(\Om)$, makes perfect sense.

In the next proposition, we state the basic results concerning the existence and uniqueness of a weak solution of the initial-boundary value problem
for the homogenized system. In fact, one can say more about the continuity in time of $(T,\vth)$, as explained below.

\begin{Prop}\lb{P-HomSyst}
Under the assumptions above, any weak solution of (\ref{HomSyst}) satisfies 
$$
\d_tT\in L^2(0,\tau;H^1(\Om)')\quad\hbox{ and }\d_t\vth\in L^2(0,\tau;L^2(\Om))\,,
$$ 
and (up to modification on some negligible $t$-set) 
$$
T,\vth\in C_b([0,+\infty);L^2(\Om))\,.
$$
Moreover, there exists a unique weak solution of the system (\ref{HomSyst}). It is a solution of the partial differential equations
$$
\left\{
\ba
{}&\d_tT-\si\Dlt_xT+4\pi\si(\rho T-\vth)=0\,,
\\
&\d_t\vth+4\pi\si'(\vth-\rho T)=0\,,
\ea
\right.
$$
in the sense of distributions on $(0,+\infty)\times\Om$, and satisfies the Neumann condition
$$
\frac{\d T}{\d n}\Rstr_{(0,\tau)\times\d\Om}=0
$$
in $H^{1/2}_{00}((0,\tau)\times\d\Om)'$ for each $\tau>0$.
\end{Prop}

\smallskip
In fact, the existence of the solution of (\ref{HomSyst}) follows from Theorem \ref{T-HomLim}

\subsection{The homogenization limit}

Henceforth we assume that the material $\cB$ occupies $N$ identical spherical inclusions with radius $\eps$:
\be\lb{DefBeps}
B_\eps=\bigcup_{i=1}^NB_i\qquad\hbox{ where }B_i:=\overline{B(x_i,\eps)}\,,\quad i=1,\ldots,N\,
\ee
and henceforth denote
\be\lb{DefAeps}
A_\eps=\Om\setminus B_\eps\,.
\ee
The number of inclusions $N$ is assumed to scale as
\be\lb{ScalNeps}
N=1/\eps\,.
\ee
The inclusion centers $x_i$ are distributed so that their empirical distribution satisfies
\be\lb{DefRho}
\frac1N\sum_{i=1}^N\de_{x_i}\to\rho\scrL^3
\ee
in the weak topology of probability measures, where $\scrL^3$ designates the $3$-dimensio\-nal Lebesgue measure and
\be\lb{HypRho}
\rho\hbox{ and }1/\rho\in C_b(\overline\Om)\,,\qquad\int_\Om\rho(x)dx=1\,.
\ee
Besides, we also assume that
\be\lb{HypXi2}
\frac1N\sum_{i=1}^N|x_i|^2\le C^{in}\qquad\hbox{ for all }N\ge 1
\ee
for some positive constant $C^{in}$. Finally, we denote
\be\lb{DefReps}
r_\eps=\eps^{1/3}
\ee
and assume that the inclusion centers are chosen so that
\be\lb{HypXiXj}
|x_i-x_j|>2r_\eps\qquad\hbox{ for all }i,j=1,\ldots,N\,.
\ee

For simplicity we assume that $\rho_A,C_A$ and $\ka_A$ are constants, and define
\be\lb{DefSi}
\si=\ka_A/\rho_AC_A\,.
\ee
We further assume that $\rho_B$ and $C_B$ are scaled with $\eps$ so that $\rho_BC_B\sim\hbox{Const.}/\eps^2$, and introduce the constant
\be\lb{DefSi'}
\si'=3\ka_A/4\pi\rho_BC_B\eps^2\,.
\ee

The scaled infinite heat conductivity problem takes the form
\be\lb{ScalInfHeatCond}
\left\{
\ba
{}&\d_tT_\eps(t,x)=\si\Dlt_xT_\eps(t,x)\,,\quad &&x\in A_\eps\,,\,\,t>0\,,
\\
&\frac{\d T_\eps}{\d n}(t,x)=0\,,&&x\in\d\Om\,,\,\,t>0\,,
\\
&T_\eps(t,x)=T_{i,\eps}(t)\,,&&x\in\d B(x_i,\eps)\,,\,\,t>0\,,\,\,1\le i\le N\,,
\\
&\dot{T}_{i,\eps}(t)=\frac{\si'}{\eps}\int_{\d B(x_i,\eps)}\frac{\d T_\eps}{\d n}(t,x)dS(x)\,,&&t>0\,,\,\,1\le i\le N\,,
\\
&T_\eps(0,x)=T^{in}_\eps(x)\,,&&x\in\Om\,.
\ea
\right.
\ee

The initial data $T^{in}_\eps\in\cH_N$, so that $T^{in}_\eps$ is a.e. a constant in $B(x_i,\eps)$:
\be\lb{DefTinieps}
T^{in}_{i,\eps}:=\frac{3}{4\pi\eps^3}\int_{B(x_i,\eps)}T^{in}(x)dx
\ee
Then
$$
\ba
|T^{in}_\eps|^2_{\cH_N}=\rho_AC_A\int_{A_\eps}T^{in}_\eps(x)^2dx+\sum_{i=1}^N\tfrac{4\pi}3\rho_BC_B\eps^3|T^{in}_{i,\eps}|^2
\\
=\rho_AC_A\left(\int_{A_\eps}T^{in}_\eps(x)^2dx+\frac{\si}{\si'}\eps\sum_{i=1}^N|T^{in}_{i,\eps}|^2\right)
\ea
$$
We shall henceforth assume that the initial data satisfies 
$$
|T^{in}_\eps|^2_{\cH_N}=O(1)
$$
i.e. that there exists a positive constant, taken equal to $C^{in}$ for notational simplicity, such that
\be\lb{HypTineps}
\int_{A_\eps}T^{in}_\eps(x)^2dx+\frac{\si}{\si'}\eps\sum_{i=1}^N|T^{in}_{i,\eps}|^2\le C^{in}\quad\hbox{ for all }\eps>0\,.
\ee

\begin{Thm}\lb{T-HomLim}
Assume that (\ref{ScalNeps}) holds, that the distribution of inclusion centers satisfies (\ref{DefRho}) and (\ref{HypXi2}), that the volumetric heat capacity 
of the material in the inclusions scales as prescribed in (\ref{DefSi'}), and that the initial data $T^{in}_\eps$ satisfies the bound (\ref{HypTineps}). Assume
further that
$$
T^{in}_\eps\to T^{in}\qquad\hbox{Êin }L^2(\Om)\hbox{ weak as }\eps\to 0
$$
while\footnote{The notation $\cM_b(\Om)$ designates the set of bounded (signed) Radon measures on $\Om$.}
$$
\frac1N\sum_{i=1}^NT^{in}_{i,\eps}\de_{x_i}\to\vth^{in}\hbox{ in }\cM_b(\Om)\hbox{ weak-* as }\eps\to 0\,.
$$
Let $T_\eps\in C([0,+\infty);\cH_N)\cap L^2(0,\tau,\cV_N)$ for all $\tau>0$ be the weak solution of the scaled infinite heat conductivity problem 
(\ref{ScalInfHeatCond}). Then, in the limit as $\eps\to 0$,
$$
T_\eps\to T\quad\left\{\ba{}&\hbox{ Êin }L^2(0,\tau;H^1(\Om))\hbox{ weak for all }\tau>0\\&\hbox{ and in }L^\infty([0,+\infty);L^2(\Om))\hbox{ weak-*,}\ea\right.
$$
and
$$
\vth_\eps:=\frac1N\sum_{i=1}^NT_{i,\eps}\de_{x_i}\to\vth\hbox{ in }L^\infty([0,+\infty);\cM_b(\Om))\hbox{ weak-*}
$$
where
$$
T_{i,\eps}:=\frac{3}{4\pi\eps^3}\int_{B(x_i,\eps)}T_\eps(t,x)dx\,.
$$
Besides
$$
T\in C_b([0,+\infty);L^2(\Om))\times L^2(0,\tau;H^1(\Om))\hbox{ for each }\tau>0
$$
while 
$$
\vth\in C_b([0,+\infty);L^2(\Om))\,.
$$
Finally, the pair $(T,\vth)$ is the unique weak solution of the homogenized system (\ref{HomSyst}) with initial condition
$$
T\rstr_{t=0}=T^{in}\,,\qquad\vth\rstr_{t=0}=\vth^{in}\,.
$$
\end{Thm}


\section{Proofs of Propositions \ref{P-FinCond}, \ref{P-VarFormInf} and \ref{P-InfCond}}


\begin{proof}[Proof of Proposition \ref{P-FinCond}]
Consider the Hilbert spaces $\cH=L^2(\Om)$ and $\cV=H^1(\Om)$ equipped with the inner products
$$
\ba
(u|v)_{\cH}&:=\int_\Om u(x)v(x)\rho(x)C(x)dx\,,
\\
(u|v)_{\cV}&:=\int_\Om(u(x)v(x)+\grad u(x)\cdot\grad v(x))\rho(x)C(x)dx\,.
\ea
$$

Let $a$ be the bilinear form defined on $\cV\times\cV$ by
$$
a(u,v)=\int_\Om\ka(x)\grad_xu(x)\cdot\grad_xv(x)dx\,;
$$
observe that
$$
|a(u,v)|\le\frac{\ka_M}{\rho_mC_m}(u|u)_{\cV}^{1/2}(v|v)_{\cV}^{1/2}
$$
while
$$
a(u,u)\ge\frac{\ka_m}{\rho_MC_M}((u|u)_{\cV} -(u|u)_{\cH})\,.
$$

By Theorem X.9 in \cite{Brezis}, there exists a unique $T\in L^2(0,\tau;\cV)\cap C_b([0,\tau];\cH)$ such that $\rho C\d_tT\in L^2(0,\tau;\cV')$ for each
$\tau>0$ such that the linear functional
$$
L(t):\,w\mapsto\d_t(T(t,\cdot)|w)_{\cH}+a(T(t,\cdot),w)=\la\rho C\d_tT(t,\cdot),w\ra_{\cV',\cV}+a(T(t,\cdot),w)
$$
satisfies
$$
\la L(t),w\ra_{\cV',\cV}=0\hbox{ for a.e. }t\in[0,+\infty)
$$
for all $w\in\cV$. Equivalently, $T$ is the unique weak solution of (\ref{HeatNeum}).

By Lemma \ref{L-Lempp}, this linear functional satisfies $L(t)=0$ for a.e. $t\in[0,+\infty)$. In particular
$$
0=\la L(s),T(s,\cdot)\ra_{\cV',\cV}=\la\rho C\d_tT(s,\cdot),T(s,\cdot)\ra_{\cV',\cV}+a(T(s,\cdot),T(s,\cdot))
$$
for a.e. $s\in[0,+\infty)$. Integrating in $s\in[0,t]$ and applying statement b) of Lemma \ref{L-LemEnerg} give the ``energy identity''.
\end{proof}

\begin{proof}[Proof of Proposition \ref{P-VarFormInf}]
Specializing (\ref{VarInfCond}) to the case where $w\in C^\infty_c(A)$ implies (\ref{HeatEqA}). In particular, the vector field
$$
(0,\tau)\times A\ni(t,x)\mapsto(\rho_A(x)C_A(x)T(t,x),-\ka_A(x)\grad_xT(t,x))
$$
is divergence free in $(0,\tau)\times A$. Applying statement b) in Lemma \ref{L-NormTr} shows that, for each $w\in\cV_N$, one has
$$
\ba
0=\frac{d}{dt}\int_\Om\rho(x)C(x)T(t,x)w(x)dx+\int_A\ka_A(x)\grad_xT(t,x)\cdot\grad w(x)dx&
\\
=\frac{d}{dt}\int_A\rho_A(x)C_A(x)T(t,x)w(x)dx+\sum_{i=1}^N\b_iw_i\dot{T}_i(t)&
\\
+\int_A\ka_A(x)\grad_xT(t,x)\cdot\grad w(x)dx&
\\
=\sum_{i=1}^Nw_i\left(\b_i\dot{T}_i(t)-\La\ka_A\frac{\d T}{\d n}(t,\cdot)\Rstr_{\d B_i},1\Ra_{H^{-1/2}(\d B_i),H^{1/2}(\d B_i)}\right)&
\\
+\La\ka_A\frac{\d T}{\d n}(t,\cdot)\Rstr_{\d\Om},w\rstr_{\d\Om}\Ra_{H^{-1/2}(\d\Om),H^{1/2}(\d\Om)}&\,,
\ea
$$
where
$$
w_i:=\frac1{|B_i|}\int_{B_i}w(y)dy\,,\quad i=1,\ldots,N\,.
$$
Since this is true for all $w\in\cV_N$, and therefore for all $(w_1,\ldots,w_N)\in\bR^N$, one concludes that
$$
\b_i\dot{T}_i-\La\ka_A\frac{\d T}{\d n}\Rstr_{\d B_i},1\Ra_{H^{-1/2}(\d B_i),H^{1/2}(\d B_i)}=0
$$
in $H^{-1}((0,\tau))$ for all $i=1,\ldots,N$, and
$$
\ka_A\frac{\d T}{\d n}\Rstr_{\d\Om}=0
$$
in $H^{1/2}_{00}((0,\tau)\times\d\Om)'$.

Conversely, if $T$ satisfies (\ref{HeatEqA}), (\ref{NeumCond2}) and (\ref{TransmCondi}), the equality above shows that (\ref{VarInfCond}) holds.
\end{proof}

\begin{proof}[Proof of Proposition \ref{P-InfCond}]
Let $b$ be the bilinear form defined on $\cV_N\times\cV_N$ by
$$
b(u,v)=\int_A\ka_A(x)\grad_xu(x)\cdot\grad_xv(x)dx\,;
$$
observe that
$$
|b(u,v)|\le\frac{\ka_M}{\rho_mC_m}(u|u)_{\cV_N}^{1/2}(v|v)_{\cV_N}^{1/2}
$$
while
$$
b(u,u)\ge\frac{\ka_m}{\rho_MC_M}((u|u)_{\cV_N} -(u|u)_{\cH_N})\,.
$$

By the same argument as in the proof of Proposition \ref{P-FinCond}, for each $T^{in}\in\cH_N$, there exists a unique weak solution of (\ref{InfHeatCond}), 
and this solution satisfies the energy identity in the statement of Proposition \ref{P-InfCond}.
\end{proof}


\section{Proof of Theorem \ref{T-InfCondLim}}


We keep the notation used in the proof of Proposition \ref{P-FinCond}, especially with the same definitions of $a,b,\cH$ and $\cV$.

For each $\eta>0$, the weak solution $T_\eta$ of (\ref{HeatNeum}) satisfies the energy identity
$$
\ba
\tfrac12\int_\Om\rho(x)C(x)T_\eta(t,x)^2dx&+\int_0^t\int_A\ka_A(x)|\grad_xT_\eta(s,x)|^2dxds
\\
&+\frac1\eta\int_0^t\int_B\ka_B(x)|\grad_xT_\eta(s,x)|^2dxds
\\
&=
\tfrac12\int_\Om\rho(x)C(x)T^{in}(x)^2dx\,.
\ea
$$
Hence, for $\eta\in(0,1)$, one has
$$
|T_\eta(t,\cdot)|^2_{\cH}\le|T^{in}|^2_{\cH}\quad\hbox{ and }\int_0^\infty|\grad_xT_\eta(t,\cdot)|^2_{\cH}dt\le\tfrac{\rho_MC_M}{2\ka_m}|T^{in}|^2_{\cH}\,.
$$
Applying the Banach-Alaoglu theorem shows that the family $T_\eta$ is relatively compact in $L^\infty([0,+\infty);\cH)$ weak-* and in $L^2([0,+\infty);\cV)$ 
weak. Let $T$ be a limit point of $T_\eta$; passing to the limit in the energy identity above shows that, by convexity and weak limit,
$$
\int_0^\infty\int_B|\grad_xT(t,x)|^2dxdt=0\,.
$$
Thus the function $x\mapsto T(t,x)$ is constant on $B_i$ for $i=1,\ldots,N$ for a.e. $t\ge 0$ and $T\in L^\infty([0,+\infty);\cH_N)\cap L^2(0,\tau;\cV_N)$.

Write the variational formulation of (\ref{HeatNeum}) for a test function $w\in\cV_N\subset\cV$:
$$
\frac{d}{dt}(T_\eta|w)_{\cH}+a(T_\eta,w)=0\quad\hbox{ in }L^2([0,\tau])\hbox{ for all }\tau>0\,.
$$
Passing to the limit in a subsequence of $T_\eta$ converging to $T$ in $L^\infty([0,+\infty);\cH)$ weak-* and in $L^2(0,\tau;\cV)$ weak, one finds that
$$
\ba
a(T_\eta,w)&=\int_A\ka_A(x)\grad_xT_\eta(t,x)\cdot\grad w(x)dx+\frac1\eta\int_B\ka_B(x)\grad_xT_\eta(t,x)\cdot\grad w(x)dx
\\
&=\int_A\ka_A(x)\grad_xT_\eta(t,x)\cdot\grad w(x)dx
\\
&\to\int_A\ka_A(x)\grad_xT(t,x)\cdot\grad w(x)dx=b(T,w)\hbox{ weakly in }L^2([0,\tau])
\ea
$$
since $\grad_xT_\eta\to\grad_xT$ weakly in $L^2([0,\tau]\times\Om)$. (The second equality above come from the fact that $\grad w=0$ on $B$ since 
$w\in\cV_N$.) On the other hand, for each $w\in\cV_N$
$$
\ba
\int_0^\tau\left|\frac{d}{dt}(T_\eta(t,\cdot)|w)_{\cH}\right|^2dt&=\int_0^\tau|a(T_\eta(t,\cdot),w)|^2dt
\\
&=\int_0^\tau\left|\int_A\ka_A(x)\grad_xT_\eta(t,x)\cdot\grad w(x)dx\right|^2dt
\\
&\le\int_0^\tau\int_A\ka_A(x)|\grad_xT_\eta(t,x)|^2dxdt\int_A\ka_A(x)|\grad w(x)|^2dx
\\
&\le\tfrac{\ka_M}{2\rho_mC_m}|T^{in}|_{\cH}^2|w|^2_{\cV}
\ea
$$
while
$$
(T_\eta|w)_{\cH}\to(T|w)_{\cH}\hbox{ in }L^\infty([0,+\infty))\hbox{ weak-*}\,.
$$

Therefore, for each $w\in\cV_N$, one has
$$
\frac{d}{dt}(T|w)_{\cH}+b(T,w)=0\quad\hbox{ in }L^2([0,\tau])\hbox{ for all }\tau>0\,,
$$
which implies in particular that
$$
\rho C\d_tT\in L^2(0,\tau;\cV'_N)\,,
$$
and therefore $T\in C_b([0,+\infty);\cH_N)$ by statement a) of Lemma \ref{L-LemEnerg}. Besides, by the Ascoli-Arzela theorem, 
$$
(T_\eta(t,\cdot)|w)_{\cH}\to(T(t,\cdot)|w)_{\cH}\quad\hbox{ uniformly in }t\in[0,\tau]\hbox{ for all }\tau>0\,.
$$
In particular
$$
(T_\eta(0,\cdot)|w)_{\cH}=(T^{in}|w)_{\cH}\to(T(0,\cdot)|w)_{\cH}
$$
so that
$$
T(0,\cdot)=T^{in}\,.
$$

In other words $T$ is the weak solution of (\ref{InfHeatCond}) with initial data $T^{in}$ --- the uniqueness of the weak solution following from Proposition
\ref{P-InfCond}. By compactness of the family $T_\eta$ and uniqueness of the limit point, we conclude that
$$
T_\eta\to T\hbox{ in }L^\infty([0,+\infty);\cH)\hbox{ weak-* and in }L^2(0,\tau;\cV)\hbox{ weak }
$$
as $\eta\to 0$. 

The energy identities in Propositions \ref{P-FinCond} and \ref{P-InfCond} are recast in the form
$$
\ba
\tfrac12\int_\Om\rho(x)C(x)T_\eta(t,x)^2dx+\int_0^t\int_A\ka_A(x)|\grad_xT_\eta(s,x)|^2dxds&
\\
\\
+\frac1\eta\int_0^t\int_B\ka_B(x)|\grad_xT_\eta(s,x)|^2dxds&
\\
=\tfrac12\int_\Om\rho(x)C(x)T^{in}(x)^2dx&\,,
\ea
$$
and
$$
\ba
\tfrac12\int_\Om\rho(x)C(x)T(t,x)^2dx+\int_0^t\int_A\ka_A(x)|\grad_xT(s,x)|^2dxds&
\\
=\tfrac12\int_\Om\rho(x)C(x)T^{in}(x)^2dx&\,.
\ea
$$
(Notice that the condition $T^{in}\in\cH_N$ is essential in order that
$$
\tfrac12\int_\Om\rho(x)C(x)T^{in}(x)^2dx=\tfrac12\int_A\rho_A(x)C_A(x)T^{in}(x)^2dx+\tfrac12\sum_{i=1}^N\b_i|T_i^{in}|^2\,;
$$
likewise
$$
\tfrac12\int_\Om\rho(x)C(x)T(t,x)^2dx=\tfrac12\int_A\rho_A(x)C_A(x)T(t,x)^2dx+\tfrac12\sum_{i=1}^N\b_iT_i(t)^2
$$
since $T(t,\cdot)\in\cH_N$ for all $t>0$.)

On the other hand, by convexity and weak convergence
$$
\tfrac12\int_\Om\rho(x)C(x)T(t,x)^2dx\le\varliminf_{\eta\to 0^+}\tfrac12\int_\Om\rho(x)C(x)T_\eta(t,x)^2dx\hbox{ for all }t>0\,,
$$
and
$$
\int_0^t\int_A\ka_A(x)|\grad_xT(s,x)|^2dxds\le\varliminf_{\eta\to 0^+}\int_0^t\int_A\ka_A(x)|\grad_xT_\eta(s,x)|^2dxds\,.
$$
We conclude from the energy identities recalled above that
$$
\tfrac12\int_\Om\rho(x)C(x)T_\eta(t,x)^2dx\to\tfrac12\int_\Om\rho(x)C(x)T(t,x)^2dx\hbox{ for all }t>0\,,
$$
while
$$
\left\{
\ba
{}&\int_0^t\int_A\ka_A(x)|\grad_xT_\eta(s,x)|^2dxds\to\int_0^t\int_A\ka_A(x)|\grad_xT(s,x)|^2dxds\,,
\\
&\frac1\eta\int_0^t\int_B\ka_B(x)|\grad_xT_\eta(s,x)|^2dxds\to 0\,,
\ea
\right.
$$
for all $t>0$. 

Therefore
$$
T_\eta\to T\hbox{ and }\grad_xT_\eta\to\grad_xT\hbox{ strongly in }L^2([0,\tau]\times\Om)
$$
as $\eta\to 0$. 


\section{Proof of Proposition \ref{P-HomSyst}}


Since
$$
\ba
\frac{d}{dt}\int_\Om T(t,x)\phi(x)dx&+\si\int_\Om\grad_xT(t,x)\cdot\grad\phi(x)dx
\\
&+4\pi\si\int_\Om(\rho(x)T(t,x)-\vth(t,x))\phi(x)dx=0
\\
\frac{d}{dt}\int_\Om\vth(t,x)\psi(x)dx&+4\pi\si'\int_\Om(\vth(t,x)-\rho(x)T(t,x))\psi(x)dx=0
\ea
$$
with 
$$
T\in L^2(0,\tau;H^1(\Om))\quad\hbox{ and }\vth\in L^2(0,\tau;L^2(\Om))
$$
one has
$$
\left|\frac{d}{dt}\int_\Om T(t,x)\phi(x)dx\right|\le(4\pi+1)\si(\|\rho\|_{L^\infty}\|T(t,\cdot)\|_{H^1(\Om)}+\|\vth\|_{L^2(\Om)})\|\phi\|_{H^1(\Om)}
$$
and
$$
\left|\frac{d}{dt}\int_\Om\vth(t,x)\psi(x)dx\right|\le4\pi\si'(\|\rho\|_{L^\infty}\|T(t,\cdot)\|_{L^2(\Om)}+\|\vth\|_{L^2(\Om)})\|\psi\|_{L^2(\Om)}
$$
so that the linear functionals
$$
\phi\mapsto\frac{d}{dt}\int_\Om T(t,x)\phi(x)dx\quad\hbox{ and }\psi\mapsto\frac{d}{dt}\int_\Om\vth(t,x)\psi(x)dx
$$
are continuous on $H^1(\Om)$ and on $L^2(\Om)$ respectively with values in $L^2([0,\tau])$. Therefore
$$
\d_tT\in L^2(0,\tau;H^1(\Om)')\quad\hbox{ and }\d_t\vth\in L^2([0,\tau]\times\Om)
$$
for each $\tau>0$. Since $T\in L^2(0,\tau;H^1(\Om))$ and $\vth\in L^2(0,\tau;L^2(\Om))$, this implies that
$$
T\hbox{ and }\vth\in C_b(\bR_+;L^2(\Om))\,.
$$

Since the system (\ref{HomSyst}) is linear, proving uniqueness reduces to proving that the only weak solution of (\ref{HomSyst}) satisfying the initial
condition $T^{in}=\vth^{in}=0$ is the trivial solution $T=\vth=0$. 

By Lemma \ref{L-Lempp}, taking $\phi(x)=T(t,x)$ and $\psi(x)=\tfrac{\si}{\si'}\vth(t,x)/\rho(x)$, one has
$$
\ba
\la\d_tT(t,\cdot),T(t,\cdot)\ra_{H^1(\Om)',H^1(\Om)}&+\si\int_\Om|\grad_xT(t,x)|^2dx
\\
&+4\pi\si\int_\Om(\rho(x)T(t,x)-\vth(t,x))T(t,x)dx=0\,,
\\
\tfrac{\si}{\si'}\int_\Om\frac1{\rho(x)}\vth(t,x)\d_t\vth(t,x)dx&+4\pi\si\int_\Om(\vth(t,x)-\rho(x)T(t,x))\frac{\vth(t,x)}{\rho(x)}dx=0\,.
\ea
$$
Adding both sides of the identities above, one finds that
$$
\ba
\la\d_tT(t,\cdot),T(t,\cdot)\ra_{H^1(\Om)',H^1(\Om)}+\tfrac{\si}{\si'}\int_\Om\frac1{\rho(x)}\vth(t,x)\d_t\vth(t,x)dx
\\
+\si\int_\Om|\grad_xT(t,x)|^2dx=0\,.
\ea
$$
Integrating both sides of the identity above on $[0,t]$ and applying Lemma \ref{L-LemEnerg} leads to
$$
\tfrac12\int_\Om T(t,x)^2dx+\tfrac{\si}{\si'}\int_\Om\frac1{\rho(x)}\vth(t,x)^2dx+\si\int_0^t\int_\Om|\grad_xT(s,x)|^2dxds=0
$$
so that $T=\vth=0$.

Specializing the variational formulation to $\phi,\psi\in C^\infty_c(\Om)$ shows that $T$ and $\vth$ satisfy
$$
\left\{
\ba
{}&\d_tT-\si\Dlt_xT+4\pi\si(\rho T-\vth)=0\,,
\\
&\d_t\vth+4\pi\si'(\vth-\rho T)=0\,,
\ea
\right.
$$
in the sense of distributions on $(0,+\infty)\times\Om$.

Finally, we apply Lemma \ref{L-NormTr} to the vector field
$$
(t,x)\mapsto((T(t,x)+\tfrac{\si}{\si'}\vth(t,x)),-\si\grad_xT(t,x))\,.
$$
Indeed,
$$
T+\tfrac{\si}{\si'}\vth\in C_b([0,+\infty);L^2(\Om))\quad\hbox{ and }\grad_xT\in L^2([0,\tau]\times\Om)
$$
for each $\tau>0$. By linear combination of the two partial differential equations in (\ref{HomSyst}), one has 
$$
\d_t(T+\tfrac{\si}{\si'}\vth)+\Div_x(-\si\grad_xT)=0
$$
in the sense of distributions on $(0,+\infty)\times\Om$, while
$$
\frac{d}{dt}\int_\Om(T(t,x)+\tfrac{\si}{\si'}\vth(t,x))\phi(x)dx-\si\int_\Om\grad_xT(t,x)\cdot\grad\phi(x)dx=0
$$
for each $\phi\in H^1(\Om)$. Therefore
$$
\frac{\d T}{\d n}\Rstr_{(0,\tau)\times\d\Om}=0
$$
in $H^{1/2}_{00}((0,\tau)\times\d\Om)'$ for each $\tau>0$.


\section{Proof of the homogenization limit}


\begin{proof}[Proof of Theorem \ref{T-HomLim}] The proof is decomposed in several steps and involves several auxiliary lemmas whose proofs belong
to the next section.

\smallskip
\noindent
\textit{Step 1: uniform bounds.}

The energy identity for the scaled infinite conductivity problem is
$$
\ba
\tfrac12\int_{A_\eps}T_\eps(t,x)^2dx+\tfrac12\eps\sum_{i=1}^N\frac{\si}{\si'}T_{i,\eps}(t)^2+\si\int_0^t\int_{A_\eps}|\grad_xT_\eps(s,x)|^2dxds
\\
=\tfrac12\int_{A_\eps}T^{in}_\eps(x)^2dx+\tfrac12\eps\sum_{i=1}^N\frac{\si}{\si'}|T^{in}_{i,\eps}|^2
\ea
$$
for all $t\ge 0$ and $\eps>0$. 

As a first consequence of this energy identity, the function $T_\eps\in C_b([0,+\infty);\cH_N)$ satisfies the bounds
$$
\ba
\|T_\eps(t,\cdot)\|^2_{\cH_N}&=\int_\Om\rho(x)C(x)T_\eps(t,x)^2dx
\\
&=\rho_AC_A\int_{A_\eps}T_\eps(t,x)^2dx+\tfrac{4\pi}3\eps^3\rho_BC_B\sum_{i=1}^NT_{i,\eps}(t)^2
\\
&\le
\rho_AC_A\left(\int_{A_\eps}T^{in}_\eps(x)^2dx+\eps\sum_{i=1}^N\frac{\si}{\si'}|T^{in}_{i,\eps}|^2\right)\le\rho_AC_AC^{in}
\ea
$$
and
$$
\si\int_0^t\int_{A_\eps}|\grad_xT_\eps(s,x)|^2dxds
	\le\tfrac12\int_{A_\eps}T^{in}_\eps(x)^2dx+\tfrac12\eps\sum_{i=1}^N\frac{\si}{\si'}|T^{in}_{i,\eps}|^2\le\tfrac12 C^{in}
$$
since $T_\eps(t,x)=T_{i,\eps}(t)$ for a.e. $x\in B(x_i,\eps)$ and all $i=1,\ldots,N$. 

A second consequence of the same energy identity is that
$$
\eps\sum_{i=1}^N\frac{\si}{\si'}T_{i,\eps}(t)^2\le\int_{A_\eps}T^{in}_\eps(x)^2dx+\eps\sum_{i=1}^N\frac{\si}{\si'}|T^{in}_{i,\eps}|^2\le C^{in}
$$
for all $t\in[0,+\infty)$ and $\eps>0$. To the weak solution $T_\eps$ of the scaled infinite conductivity problem we associate the empirical measure
$$
\mu_\eps(t,dxd\th):=\frac1N\sum_{i=1}^N\de_{x_i}\otimes\de_{T_{i,\eps}(t)}\,,\qquad N=1/\eps\,.
$$
Accordingly, we denote
$$
\mu_\eps^{in}(dxd\th):=\frac1N\sum_{i=1}^N\de_{x_i}\otimes\de_{T^{in}_{i,\eps}}\,.
$$
The estimate above is recast as
$$
\iint_{\Om\times\bR}\th^2\mu_\eps(t,dxd\th)=\eps\sum_{i=1}^NT_{i,\eps}(t)^2\le\frac{\si'}{\si}C^{in}\,.
$$
On the other hand, by assumption (\ref{HypXi2})
$$
\iint_{\Om\times\bR}|x|^2\mu_\eps(t,dxd\th)=\frac1N\sum_{i=1}^N|x_i|^2\le C^{in}\,.
$$

\smallskip
\noindent
\textit{Step 2: compactness properties.}

These uniform bounds obviously imply that the family $T_\eps$ is relatively compact in $L^\infty([0,+\infty);L^2(\Om))$ weak-* and in $L^2(0,\tau;H^1(\Om))$ 
weak for all $\tau>0$. Likewise the family $(1+|x|^2+\th^2)\mu_\eps$ is relatively compact in $L^\infty([0,+\infty);\cM_b(\Om\times\bR))$ viewed as the dual 
of the Banach space\footnote{If $X$ is a locally compact space, the notation $C_0(X)$ designates the set of real-valued continuous functions $f$ defined on 
$X$ such that $f$ converges to $0$ at infinity. This is a Banach space for the norm $\|f\|=\sup_{x\in X}|f(x)|$.} $L^1([0,+\infty);C_0(\Om\times\bR))$ equipped 
with the weak-* topology.

Henceforth, we denote by $(T,\mu)$ a limit point of the family $(T_\eps,\mu_\eps)$ as $\eps\to 0$. Define
$$
\rho(t,\cdot):=\int_{\bR}\mu(t,\cdot,d\th)\,,\qquad\vth(t,\cdot):=\int_\bR\th\mu(t,\cdot,d\th)\,.
$$

Next we return to the energy identity in step 1 recast as follows
\be\lb{EnergyReturn}
\ba
\tfrac12\int_\Om T_\eps(t,x)^2dx&+(\tfrac{\si}{\si'}-\tfrac43\pi\eps^2)\iint_{\Om\times\bR}\tfrac12\th^2\mu_\eps(t,dxd\th)
\\
&+\si\int_0^t\int_{A_\eps}|\grad_xT_\eps(s,x)|^2dxds
\\
&=\tfrac12\int_{A_\eps}T^{in}_\eps(x)^2dx+\tfrac12\eps\sum_{i=1}^N\frac{\si}{\si'}|T^{in}_{i,\eps}|^2\,,
\ea
\ee
so that
$$
(\tfrac{\si}{\si'}-\tfrac43\pi\eps^2)\iint_{\Om\times\bR}\th^2\mu_\eps(t,dxd\th)\le C^{in}\,.
$$
Thus, for each $R>0$, using $(x,\th)\mapsto\min(\th^2,R)$ as test function and the weak-* convergence of the family of measures $(1+|x|^2+\th^2)\mu_\eps$,
passing to the limit in each side of the inequality above, we get
$$
\tfrac{\si}{\si'}\iint_{\Om\times\bR}\min(\th^2,R)\mu(t,dxd\th)\le C^{in}\,.
$$
Letting $R\to+\infty$, by monotone convergence
$$
\tfrac{\si}{\si'}\iint_{\Om\times\bR}\th^2\mu(t,dxd\th)\le C^{in}\,.
$$
By the Cauchy-Schwarz inequality
$$
\vth(t,\cdot)^2=\left(\int_\bR\th\mu(t,\cdot,d\th)\right)^2\le\int_\bR\mu(t,\cdot,d\th)\int_\bR\th^2\mu(t,\cdot,d\th)=\rho\int_\bR\th^2\mu(t,\cdot,d\th)
$$
so that
$$
\int_\Om\vth(t,x)^2dx\le\int_\Om\rho(x)\int_\bR\th^2\mu(t,dxd\th)\le\frac{\si'}{\si}C^{in}\|\rho\|_{L^\infty(\Om)}\,.
$$

Thus, going back to (\ref{EnergyReturn}), we conclude that, for each $\tau>0$,
$$
T\in L^\infty([0,+\infty);L^2(\Om))\cap L^2(0,\tau;H^1(\Om))\quad\hbox{ and }\vth\in L^\infty([0,+\infty);L^2(\Om))\,.
$$

\smallskip
\noindent
\textit{Step 3: passing to the limit in the variational formulation.}

Start from the variational formulation of the scaled infinite conductivity problem: for each $\Phi_\eps\in\cV_N$
$$
\ba
\frac{d}{dt}\left(\int_{A_\eps}T_\eps(t,x)\Phi_\eps(x)dx+\tfrac{3\si}{4\pi\si'}\frac1{\eps^2}\int_{B_\eps}T_\eps(t,x)\Phi_\eps(x)dx\right)
\\
+\si\int_{A_\eps}\grad_xT_\eps(t,x)\cdot\grad\Phi_\eps(x)dx=0
\ea
$$
for a.e. $t\in[0,+\infty)$. 

Since $T_\eps(t,\cdot)\in\cV_N$, assuming that $\Phi_\eps\in\cV_N\cap C_b(\overline{\Om})$, 
$$
\tfrac{3\si}{4\pi\si'}\frac1{\eps^2}\int_{B_\eps}T_\eps(t,x)\Phi_\eps(x)dx=\tfrac{\si}{\si'}\eps\sum_{i=1}^NT_{i,\eps}(t)\Phi_{i,\eps}
=
\frac{\si}{\si'}\iint_{\Om\times\bR}\Phi_\eps(x)\th\mu_\eps(t,dxd\th)\,.
$$

On the other hand
$$
\ba
\int_{A_\eps}T_\eps(t,x)\Phi_\eps(x)dx&=\int_{\Om}T_\eps(t,x)\Phi_\eps(x)dx-\int_{B_\eps}T_\eps(t,x)\Phi_\eps(x)dx
\\
&=\int_{\Om}T_\eps(t,x)\Phi_\eps(x)dx-\tfrac{4\pi}3\eps^3\sum_{i=1}^NT_{i,\eps}(t)\Phi_{i,\eps}(t)
\\
&=\int_{\Om}T_\eps(t,x)\Phi_\eps(x)dx-\tfrac{4\pi}3\eps^2\iint_{\Om\times\bR}\Phi_\eps(x)\th\mu_\eps(t,dxd\th)
\ea
$$
so that
$$
\ba
\left|\int_{\Om}T_\eps(t,x)\Phi_\eps(x)dx-\int_{A_\eps}T_\eps(t,x)\Phi_\eps(x)dx\right|&
\\
\le\tfrac{2\pi}3\eps^2\|\Phi_\eps\|_{L^\infty(\Om)}\iint_{\Om\times\bR}(1+\th^2)\mu_\eps(t,dxd\th)&
\\
\le\tfrac{2\pi}3\eps^2\|\Phi_\eps\|_{L^\infty(\Om)}(1+\tfrac{\si'}{\si}C^{in})&\,.
\ea
$$

Finally
$$
\int_{A_\eps}\grad_xT_\eps(t,x)\cdot\grad\Phi_\eps(x)dx=\int_{\Om}\grad_xT_\eps(t,x)\cdot\grad\Phi_\eps(x)dx\,.
$$

We shall pass to the limit in the variational formulation above for two different classes of test functions $\Phi_\eps$.

\newpage
\noindent
\textit{Step 4: first class of test functions.}

Let $\phi\in C^1_c(\overline\Om)$. By the mean value theorem 
$$
|\phi(x)-\phi(x_i)|\le\eps\|D\phi\|_{L^\infty}
$$
so that $\phi$ ``almost'' belongs to $\cV_N$ --- but in general does not belong to $\cV_N$. This difficulty is fixed by the following procedure.

For each $\psi\in C(\overline{B(0,\eps)})$, define $\chi[\psi]$ to be the solution of
\be\lb{PbmChi}
\left\{
\ba
{}&\Dlt\chi[\psi](z)=0\,,&&\quad\eps<|z|<r_\eps\,,
\\
&\chi[\psi](z)=\psi(z)\,,&&\quad|z|\le\eps\,,
\\
&\chi[\psi](z)=0\,,&&\quad|z|=r_\eps\,.
\ea
\right.
\ee
Define
$$
\cQ_\eps(x):=\sum_{i=1}^N\chi[\phi(x_i+\cdot)-\phi(x_i)](x-x_i)\,,
$$
and let 
$$
\Phi_\eps(x):=\phi(x)-\cQ_\eps(x)\,.
$$

\begin{Lem}\lb{L-Qeps}
For each $\eps>0$, one has
$$
\|\cQ_\eps\|_{L^\infty(\Om)}\le 2\|\phi\|_{L^\infty(\Om)}\,.
$$
Besides
$$
\cQ_\eps\to 0\hbox{ in }H^1(\Om)\hbox{ strong }
$$
as $\eps\to 0$.
\end{Lem}

The proof of this lemma is postponed to the end of this section. Taking this for granted, one has
$$
\Phi_\eps\to\phi\hbox{ in }H^1(\Om)\hbox{ strong }
$$
as $\eps\to 0$. Therefore
$$
\ba
\int_{A_\eps}\grad_xT_\eps(t,x)\cdot\grad\Phi_\eps(x)dx=\int_{\Om}\grad_xT_\eps(t,x)\cdot\grad\Phi_\eps(x)dx&
\\
\to\int_{\Om}\grad_xT(t,x)\cdot\grad\phi(x)dx&\hbox{ weakly in }L^2([0,+\infty))
\ea
$$
as $\eps\to 0$. 

On the other hand
$$
\ba
\int_\Om T_\eps(t,x)\Phi_\eps(x)dx&=\int_\Om T_\eps(t,x)\phi(x)dx-\int_\Om T_\eps(t,x)\cQ_\eps(x)dx
\\
&\to\int_\Om T(t,x)\phi(x)dx\hbox{ in }L^\infty([0,+\infty))\hbox{ weak-*}
\ea
$$
as $\eps\to 0$ since
$$
\left|\int_\Om T_\eps(t,x)\cQ_\eps(x)dx\right|\le\|T_\eps(t,\cdot)\|_{L^2}\|\cQ_\eps\|_{L^2}\,.
$$
Indeed
$$
\sup_{t\ge 0}\|T_\eps(t,\cdot)\|_{L^2}<\infty\,,\quad\hbox{ while }\|\cQ_\eps\|_{L^2}\to 0\hbox{ as }\eps\to 0
$$
by Lemma \ref{L-Qeps}.

Finally
$$
\iint_{\Om\times\bR}\Phi_\eps(x)\th\mu_\eps(t,dxd\th)=\iint_{\Om\times\bR}\phi(x)\th\mu_\eps(t,dxd\th)
$$
since $\phi(x_i)=\Phi_\eps(x_i)$ for $i=1,\ldots,N$, so that
$$
\iint_{\Om\times\bR}\Phi_\eps(x)\th\mu_\eps(t,dxd\th)\to\iint_{\Om\times\bR}\phi(x)\th\mu(t,dxd\th)\hbox{ in }L^\infty([0,+\infty))\hbox{ weak-*}
$$
as $\eps\to 0$.

By construction $\Phi_\eps\in\cV_N$, so that $\Phi_\eps$ can be used as a test function  in the variational formulation. Passing to the limit in the 
variational formulation of the scaled infinite heat conductivity problem in the sense of distributions gives
$$
\ba
\frac{d}{dt}\left(\int_\Om T(t,x)\phi(x)dx+\frac{\si}{\si'}\iint_{\Om\times\bR}\phi(x)\th\mu(t,dxd\th)\right)
\\
+
\si\int_{\Om}\grad_xT(t,x)\cdot\grad\phi(x)dx=0
\ea
$$
in $L^2_{loc}([0,+\infty))$ for each $\phi\in C^1_c(\overline\Om)$.

\smallskip
\noindent
\textit{Step 5: second class of test functions}

In this step, we shall use a class of test functions $\Psi_\eps\in H^1(\Om)$ such that $\Psi_\eps\rstr_{B(x_i,\eps)}=0$ for all $i=1,\ldots,N$. Given 
$\phi\in C^1_c(\overline\Om)$, define $\Psi_\eps$ as follows:
$$
\Psi_\eps(x):=\phi(x)-\cP_\eps(x)
$$
where
$$
\cP_\eps(x):=\sum_{i=1}^N\chi[\phi(x_i+\cdot)](x-x_i)\,.
$$
We shall further decompose $\cP_\eps$ as follows:
$$
\ba
\cP_\eps(x)&=\sum_{i=1}^N\chi[\phi(x_i+\cdot)-\phi(x_i)](x-x_i)+\sum_{i=1}^N\chi[\phi(x_i)](x-x_i)
\\
&=\cQ_\eps(x)+\cR_\eps(x)\,.
\ea
$$

Likewise, one associates to the solution $T_\eps$ of the scaled infinite heat conductivity problem
$$
\Th_\eps(t,x):=T_\eps(t,x)-\cS_\eps(t,x)
$$
where
$$
\cS_\eps(t,x):=\sum_{i=1}^N\chi[T_{i,\eps}(t)](x-x_i)\,.
$$

The variational formulation for the test function $\Psi_\eps$ becomes
$$
\frac{d}{dt}\int_\Om T_\eps(t,x)\Psi_\eps(x)dx+\si\int_\Om\grad_xT_\eps(t,x)\cdot\grad\Psi_\eps(x)dx=0
$$
in $L^2_{loc}([0,+\infty))$, since $\Psi_\eps=0$ on $B(x_i,\eps)$ for all $i=1,\ldots,N$.

\begin{Lem}\lb{L-RSeps}
One has
$$
\cR_\eps\to 0\quad\hbox{ in }H^1(\Om)\hbox{ weak}\,,
$$
so that
$$
\cP_\eps\to 0\hbox{ in }H^1(\Om)\hbox{Êweak}\,,
$$
while
$$
\cS_\eps\to 0\quad\hbox{Êin }L^\infty([0,+\infty);H^1(\Om))\hbox{ weak-*}
$$
as $\eps\to 0$.
\end{Lem}

\smallskip
Taking this lemma for granted, and observing that
$$
\Supp(\cP_\eps)\subset\Supp(\phi)+\overline{B(0,r_\eps)}
$$
the Rellich compactness theorem implies that
$$
\cP_\eps\to 0\quad\hbox{ in }L^2(\Om)\hbox{ strong}
$$
as $\eps\to 0$, so that 
$$
\int_\Om T_\eps(t,x)\Psi_\eps(x)dx\to\int_\Om T(t,x)\phi(x)dx\hbox{ in }L^\infty([0,+\infty))\hbox{ weak-*}
$$
as $\eps\to 0$.

Next, decompose
$$
\ba
\int_\Om\grad_xT_\eps(t,x)\cdot\grad\Psi_\eps(x)dx&=\int_\Om\grad_x\Th_\eps(t,x)\cdot\grad\phi(x)dx
\\
&-\int_\Om\grad_x\Th_\eps(t,x)\cdot\grad\cP_\eps(x)dx
\\
&+\int_\Om\grad_x\cS_\eps(t,x)\cdot\grad\Psi_\eps(x)dx
\ea
$$
Since $\cS_\eps\to 0$ in $L^2(0,\tau;H^1_0(\Om))$ weak as $\eps\to 0$, one has
$$
\grad_x\Th_\eps=\grad_xT_\eps-\grad_x\cS_\eps\to\grad_xT\quad\hbox{Êin }L^2([0,\tau]\times\Om)\hbox{ weak}
$$
as $\eps\to 0$, so that
$$
\int_\Om\grad_x\Th_\eps(t,x)\cdot\grad\phi(x)dx\to\int_\Om\grad_x T(t,x)\cdot\grad\phi(x)dx\hbox{ in }L^2([0,\tau])\hbox{ weak}
$$
as $\eps\to 0$.

Furthermore, one has
$$
\ba
\int_\Om\grad_x\cS_\eps(t,x)\cdot\grad\Psi_\eps(x)dx&=\int_\Om\grad_x\cS_\eps(t,x)\cdot\grad\phi(x)dx
-\int_\Om\grad_x\cS_\eps(t,x)\cdot\grad\cQ_\eps(x)dx
\\
&-\sum_{i=1}^N\int_{B(x_i,r_\eps)\setminus B(x_i,\eps)}\grad\chi[T_{i,\eps}(t)](z)\cdot\grad\chi[\phi(x_i)](z)dz\,.
\ea
$$
As noticed above, $\grad_x\cS_\eps\to 0$ weakly in $L^2([0,\tau]\times\Om)$ for all $\tau>0$ as $\eps\to 0$, and therefore
$$
\int_\Om\grad_x\cS_\eps(t,x)\cdot\grad\phi(x)dx\to 0\quad\hbox{ in }L^2([0,\tau])\hbox{ weak}
$$
for all $\tau>0$ as $\eps\to 0$, while
$$
\int_\Om\grad_x\cS_\eps(t,x)\cdot\grad\cQ_\eps(x)dx\to 0\quad\hbox{ in }L^2([0,\tau])\hbox{ strong}
$$
for all $\tau>0$ as $\eps\to 0$ by Lemma \ref{L-Qeps}. 

The third term on the right hand side of the last equality is handled with the following lemma.

\begin{Lem}\lb{L-Tphi}
One has
$$
\sum_{i=1}^N\int_{B(0,r_\eps)\setminus B(0,\eps)}\grad\chi[T_{i,\eps}(t)](z)\cdot\grad\chi[\phi(x_i)](z)dz
\to 4\pi\int_{\Om\times\bR}\phi(x)\th\mu(t,dxd\th)
$$
in $L^\infty(\bR_+)$ weak-*.
\end{Lem}

Therefore
$$
\int_\Om\grad_x\cS_\eps(t,x)\cdot\grad\Psi_\eps(x)dx\to-4\pi\iint_{\Om\times\bR}\phi(x)\th\mu(t,dxd\th)
$$
in $L^2([0,\tau])$ weak as $\eps\to 0$.

It remains to treat the term
$$
\ba
\int_\Om\grad_x\Th_\eps(t,x)\cdot\grad\cP_\eps(x)dx
=
\int_\Om\grad_x\Th_\eps(t,x)\cdot\grad\cQ_\eps(x)dx
\\
+
\sum_{i=1}^N\int_{B(x_i,r_\eps)\setminus B(x_i,\eps)}\grad_x\Th_\eps(t,z)\cdot\grad\chi[\phi(x_i)](z)dz
\ea
$$
By Lemma \ref{L-Qeps}, $\cQ_\eps\to 0$ in $H^1(\Om)$ strong; by the second convergence in Lemma \ref{L-RSeps}, the family $\grad_x\cS_\eps$ is
bounded in $L^2([0,\tau]\times\Om)$ for each $\tau>0$, while $\grad_xT_\eps$  is bounded in $L^2([0,\tau]\times\Om)$ as explained in Step 1. Thus
$\grad_x\Th_\eps$  is bounded in $L^2([0,\tau]\times\Om)$ for all $\tau>0$, so that
$$
\int_\Om\grad_x\Th_\eps(t,x)\cdot\grad\cQ_\eps(x)dx\to 0\hbox{ in }L^2([0,\tau])\hbox{Ê strong}
$$
for each $\tau>0$ as $\eps\to 0$.

Next, by Green's formula
$$
\int_{B(x_i,r_\eps)\setminus B(x_i,\eps)}\grad_x\Th_\eps(t,z)\cdot\grad\chi[\phi(x_i)](z)dz
=
\int_{\d B(x_i,r_\eps)}\Th_\eps(t,z)\frac{\d\chi[\phi(x_i)]}{\d n}(z)dz
$$
since $\chi[\phi(x_i)]$ is harmonic on $B(x_i,r_\eps)\setminus B(x_i,\eps)$ and $\Th_\eps\rstr_{\d B(x_i,\eps)}=0$.

\begin{Lem}\lb{L-Cap}
For each $\phi\in C_b(\bR^3)$, one has
$$
\sum_{i=1}^N\frac{\d}{\d n}\left(\chi[\phi(x_i)](x-x_i)\right)\de_{\d B(x_i,r_\eps)}
=
-\frac{\eps r_\eps}{r_\eps^2(r_\eps-\eps)}\sum_{i=1}^N\phi(x_i)\de_{\d B(x_i,r_\eps)}
\to
-4\pi\rho\phi
$$
in $H^{-1}(\bR^3)$. We recall that $\rho\in C_b(\overline\Om)$ is defined as follows:
$$
\frac1N\sum_{i=1}^N\de_{x_i}\to\rho\scrL^3
$$
weakly in the sense of probability measures on $\overline\Om$, where $\scrL^3$ designates the $3$-dimensional Lebesgue measure.
\end{Lem}

Taking this lemma for granted, we see that
$$
\int_\Om\grad_x\Th_\eps(t,x)\cdot\grad\cR_\eps(x)dx\to-4\pi\int_\Om\rho(x)T(t,x)\phi(x)dx\hbox{ in }L^2([0,\tau])\hbox{ weak}
$$
for each $\tau>0$ as $\eps\to 0$.

Summarizing the various limits established in this step, we conclude that, for each $\phi\in C^1_c(\overline\Om)$
$$
\ba
\frac{d}{dt}\int_\Om T(t,x)\phi(x)dx+\si\int_\Om\grad_x T(t,x)\cdot\grad\phi(x)dx
+4\pi\si\int_\Om\rho(x)T(t,x)\phi(x)dx
\\-4\pi\si\iint_{\Om\times\bR}\phi(x)\th\mu(t,dxd\th)=0
\ea
$$
in $L^2_{loc}([0,+\infty))$.

\bigskip
\noindent
\textit{Step 6: initial conditions}

As explained in steps 3-4, for each $\phi\in C^1_c(\overline\Om)$, defining 
$$
\Phi_\eps=\phi-\cQ_\eps\in\cV_N
$$
one has
\be\lb{1stConsDens}
\ba
\int_{A_\eps}T_\eps(t,x)\Phi_\eps(x)dx&+\frac{\si}{\si'}\iint_{\Om\times\bR}\Phi_\eps(x)\th\mu_\eps(t,dxd\th)
\\
&\to
\int_{\Om}T(t,x)\phi(x)dx+\frac{\si}{\si'}\iint_{\Om\times\bR}\phi(x)\th\mu(t,dxd\th)
\ea
\ee
in $L^2([0,\tau])$ weak as $\eps\to 0$, while
$$
\ba
\int_0^\infty\left|\frac{d}{dt}\left(\int_{A_\eps}T_\eps(t,x)\Phi_\eps(x)dx+\frac{\si}{\si'}\iint_{\Om\times\bR}\Phi_\eps(x)\th\mu_\eps(t,dxd\th)\right)\right|^2dt
\\
\le
\si^2\|\grad\Phi_\eps\|^2_{L^2(\Om)}\int_0^\infty\|\grad_xT_\eps(t,\cdot)\|^2_{L^2(\Om)}dt
\\
\le C^{in}\si(\|\grad\phi\|_{L^2(\Om)}+o(1))^2
\ea
$$
By Ascoli-Arzela's theorem, the convergence in (\ref{1stConsDens}) is uniform on $[0,\tau]$ for each $\tau$. In particular
$$
\ba
\int_{A_\eps}T^{in}_\eps(x)\Phi_\eps(x)dx&+\frac{\si}{\si'}\iint_{\Om\times\bR}\Phi_\eps(x)\th\mu^{in}_\eps(dxd\th)
\\
&\to
\int_\Om T^{in}(x)\phi(x)dx+\frac{\si}{\si'}\iint_{\Om\times\bR}\phi(x)\th\mu^{in}(dxd\th)
\\
&=
\int_\Om T(0,x)\phi(x)dx+\frac{\si}{\si'}\iint_{\Om\times\bR}\phi(x)\th\mu(0,dxd\th)
\ea
$$
for each $\phi\in C^1_c(\overline\Om)$, so that
$$
T(0,\cdot)+\frac{\si'}{\si}\int_{\bR}\th\mu(0,\cdot,d\th)=T^{in}+\frac{\si'}{\si}\int_{\bR}\th\mu^{in}(\cdot,d\th)\,.
$$

Likewise, we have seen in step 5 that, for each $\phi\in C^1_c(\overline\Om)$, defining $\Psi_\eps$ as
$$
\Psi_\eps=\phi-\cP_\eps\,,
$$
one has
\be\lb{2ndConsDens}
\int_\Om T_\eps(t,x)\Psi_\eps(x)dx\to\int_\Om T(t,x)\phi(x)dx\hbox{ in }L^\infty([0,+\infty))\hbox{ weak-*}
\ee
as $\eps\to 0$. Besides
$$
\ba
\int_0^\infty\left|\frac{d}{dt}\int_\Om T_\eps(t,x)\Psi_\eps(x)dx\right|^2dt
\le
\si^2\|\grad \Psi_\eps\|^2_{L^2(\Om)}\int_0^\infty\int_\Om|\grad_xT_\eps(t,x)|^2dxdt
\\
\le
C^{in}\si(\|\grad\phi\|_{L^2(\Om)}+o(1))^2
\ea
$$
By the Ascoli-Arzela theorem, the convergence in (\ref{2ndConsDens}) is uniform in $[0,\tau]$ for each $\tau>0$. In particular
$$
\int_\Om T^{in}_\eps(x)\Psi_\eps(x)dx\to\int_\Om T^{in}(x)\phi(x)dx=\int_\Om T(0,x)\phi(x)dx
$$
so that
$$
T(0,\cdot)=T^{in}\,.
$$

\bigskip
\noindent
\textit{Step 7: identification of the limiting system}

In steps 4-5, we have proved that
$$
\ba
\frac{d}{dt}\int_\Om T(t,x)\Phi(x)dx+\si\int_\Om\grad_x T(t,x)\cdot\grad\Phi(x)dx
\\
+4\pi\si\int_\Om(\rho(x) T(t,x)-\vth(t,x))\Phi(x)dx=0
\ea
$$
and
$$
\ba
\frac{d}{dt}\left(\int_\Om T(t,x)\Phi(x)dx+\frac{\si}{\si'}\int_\Om\vth(t,x)\Phi(x)dx\right)
\\
+
\si\int_{\Om}\grad_xT(t,x)\cdot\grad\Phi(x)dx=0
\ea
$$
for each $\Phi\in C^1_c(\overline\Om)$. By linear combination, one finds that
$$
\ba
\frac{d}{dt}\int_\Om T(t,x)\phi(x)dx&+\si\int_\Om\grad_xT(t,x)\cdot\grad\phi(x)dx
\\
&+4\pi\si\int_\Om(\rho(x)T(t,x)-\vth(t,x))\phi(x)dx=0
\\
\frac{d}{dt}\int_\Om\vth(t,x)\psi(x)dx&+4\pi\si'\int_\Om(\vth(t,x)-\rho(x)T(t,x))\psi(x)dx=0
\ea
$$
for all $\phi,\psi\in C^1_c(\overline\Om)$. Since $T,\vth\in C_b([0,+\infty);L^2(\Om))$ and $T\in L^2(0,\tau;H^1(\Om))$ for each $\tau>0$, the identities
above hold for all $\phi,\psi\in H^1(\Om)$ by a straightforward density argument.

Thus $(T,\vth)$ is the unique weak solution of (\ref{HomSyst}) with initial data $(T^{in},\vth^{in})$.

By compactness, this implies that
$$
T_\eps\to T\hbox{ in }L^\infty([0,+\infty);L^2(\Om))\hbox{ weak-* and in }L^2(0,\tau;H^1(\Om))\hbox{ weak}\,,
$$
while
$$
\vth_\eps\to\vth\hbox{ in }L^\infty([0,+\infty);L^2(\Om))\hbox{ weak-*}
$$
without extracting subsequences.
 \end{proof}


\section{Proof of Lemmas \ref{L-Qeps}, \ref{L-RSeps}, \ref{L-Tphi}  and \ref{L-Cap}}


When $\psi=1$, the solution of the boundary value problem (\ref{PbmChi}) is given by
$$
\chi[1](z)=\frac{\eps r_\eps}{r_\eps-\eps}\left(\frac1{|z|}-\frac1{r_\eps}\right)\indc_{B(0,r_\eps)\setminus B(0,\eps)}(z)+\indc_{B(0,\eps)}(z)
$$
for all $z\in\bR^3$. In that case
$$
\|\chi[1]\|_{L^2(\bR^3)}^2=\tfrac{4\pi}3\eps^2r_\eps\,,\quad\|\grad\chi[1]\|_{L^2(\bR^3)}^2=4\pi\frac{\eps r_\eps}{r_\eps-\eps}\sim 4\pi\eps
	\quad\hbox{ as }\eps\to 0\,.
$$

\begin{proof}[Proof of Lemma \ref{L-Qeps}]
First, by the maximum principle and the mean value theorem, one has
$$
\ba
\|\chi[\phi(x_i+\cdot)-\phi(x_i)]\|_{L^\infty(\bR^3)}&\le\|\phi(x_i+\cdot)-\phi(x_i)\|_{L^\infty(\bR^3)}
\\
&\le\min(2\|\phi\|_{L^\infty(\bR^3)},\|\grad\phi\|_{L^\infty(\bR^3)}\eps)\,.
\ea
$$
Since the functions $x\mapsto\chi[\phi(x_i+\cdot)-\phi(x_i)](x-x_i)$ have disjoint supports by (\ref{HypXiXj}), one has both
$$
\|\cQ_\eps\|_{L^\infty(\Om)}\le\sup_{1\le i\le N}\|\chi[\phi(x_i+\cdot)-\phi(x_i)]\|_{L^\infty(\bR^3)}\le 2\|\phi\|_{L^\infty(\bR^3)}\,,
$$
and
$$
\ba
\|\cQ_\eps\|^2_{L^2(\Om)}&\le\sum_{i=1}^N\|\chi[\phi(x_i+\cdot)-\phi(x_i)]\|^2_{L^2(\bR^3)}
\\
&\le N|B(0,r_\eps)|\|\chi[\phi(x_i+\cdot)-\phi(x_i)]\|^2_{L^\infty(\bR^3)}
\\
&\le N\cdot\tfrac43\pi r_\eps^3\|\grad\phi\|^2_{L^\infty(\bR^3)}\eps^2=\tfrac43\pi\|\grad\phi\|^2_{L^\infty(\bR^3)}\eps^2\to 0
\ea
$$
as $\eps\to 0$. Next
$$
\ba
\|\grad\chi[\phi(x_i+\cdot)-\phi(x_i)]\|^2_{L^2(\bR^3)}&=\|\grad\chi[\phi(x_i+\cdot)-\phi(x_i)]\|^2_{L^2(B(0,\eps))}
\\
&+\|\grad\chi[\phi(x_i+\cdot)-\phi(x_i)]\|^2_{L^2(B(0,r_\eps)\setminus B(0,\eps))}\,.
\ea
$$
First
$$
\ba
\|\grad\chi[\phi(x_i+\cdot)-\phi(x_i)]\|^2_{L^2(B(0,\eps))}
&\le
\tfrac43\pi\eps^3\|\grad\chi[\phi(x_i+\cdot)-\phi(x_i)]\|^2_{L^\infty(B(0,\eps))}
\\
&=
\tfrac43\pi\eps^3\|\grad\phi\|^2_{L^\infty(B(0,\eps))}\,.
\ea
$$
Since $\chi[\phi(x_i+\cdot)-\phi(x_i)]$ is a harmonic function on $B(0,r_\eps)\setminus B(0,\eps)$, it minimizes the Dirichlet integral among functions
with the same boundary values. Thus
$$
\|\grad\chi[\phi(x_i+\cdot)-\phi(x_i)]\|^2_{L^2(B(0,r_\eps)\setminus B(0,\eps))}
\le
\|\grad\chi_{i,\eps}\|^2_{L^2(B(0,r_\eps)\setminus B(0,\eps))}
$$
where
$$
\chi_{i,\eps}(z)=\left(\phi\left(x_i+\eps\frac{z}{|z|}\right)-\phi(x_i)\right)\frac{r_\eps-|z|}{r_\eps-\eps}\,.
$$
Straightforward computations show that
$$
\ba
\grad\chi_{i,\eps}(z)&=\left(I-\frac{z\otimes z}{|z|^2}\right)\grad\phi\left(x_i+\eps\frac{z}{|z|}\right)\frac{\eps}{|z|}\frac{r_\eps-|z|}{r_\eps-\eps}
\\
&-\left(\phi\left(x_i+\eps\frac{z}{|z|}\right)-\phi(x_i)\right)\frac1{r_\eps-\eps}\frac{z}{|z|}\,,
\ea
$$
so that
$$
\ba
|\grad\chi_{i,\eps}(z)|^2&\le\left|\grad\phi\left(x_i+\eps\frac{z}{|z|}\right)\right|^2\frac{\eps^2}{|z|^2}\frac{(r_\eps-|z|)^2}{(r_\eps-\eps)^2}
\\
&+
\left(\phi\left(x_i+\eps\frac{z}{|z|}\right)-\phi(x_i)\right)^2\frac1{(r_\eps-\eps)^2}\,.
\ea
$$
Thus
$$
\|\grad\chi_{i,\eps}\|^2_{L^2(B(0,r_\eps)\setminus B(0,\eps))}\le \tfrac{8\pi}{3}\|\grad\phi\|^2_{L^\infty(\bR^3)}\eps^2r_\eps+O(\eps^3r_\eps)\,,
$$
so that
$$
\ba
\|\grad\cQ_\eps\|^2_{L^2(\Om)}&\le\sum_{i=1}^N\|\grad\chi[\phi(x_i+\cdot)-\phi(x_i)]\|^2_{L^2(\bR^3)}
\\
&\le\tfrac{4\pi}{3}\|\grad\phi\|^2_{L^\infty(\bR^3)}N(\eps^3+2\eps^2r_\eps)
\\
&=\tfrac{4\pi}{3}\|\grad\phi\|^2_{L^\infty(\bR^3)}(\eps^2+2\eps r_\eps)\to 0
\ea
$$
as $\eps\to 0$. Hence $\cQ_\eps\to 0$ in $H^1(\Om)$ strong as $\eps\to 0$.

\end{proof}

\begin{proof}[Proof of Lemma \ref{L-RSeps}]
Assume that $0<\eps<\tfrac18$. Since the functions $x\mapsto\chi[1](x-x_i)$ have disjoint supports by (\ref{HypXiXj}), one has
$$
\ba
{}&\|\cS_\eps(t,\cdot)|^2_{L^2(\Om)}\le\sum_{i=1}^NT_{i,\eps}(t)^2\|\chi[1]\|^2_{L^2(\bR^3)}
\\
&\|\grad\cS_\eps(t,\cdot)\|^2_{L^2(\Om)}\le\sum_{i=1}^NT_{i,\eps}(t)^2\|\grad\chi[1]\|^2_{L^2(\bR^3)}\,.
\ea
$$
Thus
$$
\|\cS_\eps(t,\cdot)\|^2_{L^2(\Om)}\le\tfrac{4\pi}3\eps^2r_\eps\sum_{i=1}^NT_{i,\eps}(t)^2\le\tfrac{4\pi}3\frac{\si'}{\si}C^{in}\eps r_\eps\to 0
$$
as $\eps\to 0$, while
$$
\ba
\|\grad\cS_\eps(t,\cdot)\|^2_{L^2(\Om)}&\le 4\pi\frac{\eps r_\eps}{r_\eps-\eps}\sum_{i=1}^NT_{i,\eps}(t)^2
\\
&\le 8\pi\eps\sum_{i=1}^NT_{i,\eps}(t)^2\le 8\pi\frac{\si'}{\si}C^{in}\,.
\ea
$$
Hence $\cS_\eps(t,\cdot)\to 0$ in $H^1(\Om)$ weak, uniformly in $t\ge 0$ as $\eps\to 0$.

Now for $\cR_\eps$. First
$$
\ba
\|\cR_\eps\|^2_{L^2(\Om)}&=\sum_{i=1}^N\|\chi[\phi(x_i)]\|^2_{L^2(\bR^3)}
\\
&\le\sum_{i=1}^N\tfrac43\pi\phi(x_i)^2\eps^2r_\eps\le\tfrac43\pi\|\phi\|^2_{L^\infty(\bR^3)}\eps r_\eps\to 0
\ea
$$
as $\eps\to 0$, because the functions $x\mapsto\chi[1](x-x_i)$ have disjoint supports by (\ref{HypXiXj}). By the same token
$$
\ba
\|\grad\cR_\eps\|^2_{L^2(\Om)}&=\sum_{i=1}^N\|\grad\chi[\phi(x_i)]\|^2_{L^2(\bR^3)}
\\
&\le\sum_{i=1}^N4\pi\phi(x_i)^2\frac{\eps r_\eps}{r_\eps-\eps}\le 4\pi\|\phi\|^2_{L^\infty(\bR^3)}\frac{N\eps r_\eps}{r_\eps-\eps}=O(1)
\ea
$$
as $\eps\to 0$. Thus $\cR_\eps\to 0$ in $H^1(\Om)$ weak as $\eps\to 0$.
\end{proof}

\begin{proof}[Proof of Lemma \ref{L-Tphi}]
One has
$$
\ba
\sum_{i=1}^N\int_{B(0,r_\eps)\setminus B(0,\eps)}&\grad\chi[T_{i,\eps}(t)](z)\cdot\grad\chi[\phi(x_i)](z)dz
\\
&=
\|\grad\chi[1]\|^2_{L^2(\bR^N)}\sum_{i=1}^NT_{i,\eps}(t)\phi(x_i)
\\
&=
4\pi\frac{\eps r_\eps}{r_\eps-\eps}\sum_{i=1}^NT_{i,\eps}(t)\phi(x_i)
\\
&=
4\pi\frac{r_\eps}{r_\eps-\eps}\iint_{\Om\times\bR}\phi(x)\th\mu_\eps(t,dxd\th)
\\
&\to
4\pi\iint_{\Om\times\bR}\phi(x)\th\mu(t,dxd\th)
=
4\pi\int_\Om\phi(x)\vth(t,x)dx
\ea
$$
as $\eps\to 0$.
\end{proof}

\begin{proof}[Proof of Lemma \ref{L-Cap}]
First
$$
\ba
\sum_{i=1}^N\frac{\d}{\d n}\left(\chi[\phi(x_i)](x-x_i)\right)&\de_{\d B(x_i,r_\eps)}
\\
&=
\sum_{i=1}^N\phi(x_i)\frac{x-x_i}{|x-x_i|}\cdot\grad\chi[1](x-x_i)\de_{\d B(x_i,r_\eps)}
\\
&=
-\sum_{i=1}^N\phi(x_i)\frac{\eps r_\eps}{r_\eps-\eps}\frac1{|x-x_i|^2}\de_{\d B(x_i,r_\eps)}
\\
&=
-\frac{\eps r_\eps}{r_\eps^2(r_\eps-\eps)}\sum_{i=1}^N\phi(x_i)\de_{\d B(x_i,r_\eps)}\,.
\ea
$$

Next we recall that
$$
\sum_{i=1}^N\phi(x_i)r_\eps\de_{\d B(x_i,r_\eps)}\to 4\pi\rho\phi\quad\hbox{ strongly in }H^{-1}(\bR^3)
$$
as $\eps\to 0$. This result has been proved by Cioranescu-Murat \cite{CioraMura} in the case where $x_i$ are distributed periodically; see formula
(64) and Appendix 1 in \cite{DGR1} for a proof adapted to the setting of the present paper.

With the explicit formula above and the fact that $\tfrac{\eps}{r_\eps^2(r_\eps-\eps)}\to 1$ as $\eps\to 0$, this concludes the proof of Lemma \ref{L-Cap}.
\end{proof}


\begin{appendix}

\section{Some Lemmas on Evolution Equations}


Let $\cV$ and $\cH$ be two separable Hilbert spaces such that $\cV\subset\cH$ with continuous inclusion and $\cV$ is dense in $\cH$. The Hilbert 
space $\cH$ is identified with its dual and the map 
$$
\cH\ni u\mapsto L_u\in\cV'\,,
$$
where $L_u$ is the linear functional 
$$
L_u:\cV\ni v\mapsto(u|v)_{\cH}\in\bR\,,
$$
identifies $\cH$ with a dense subspace of $\cV'$.

\begin{Lem}\lb{L-LemEnerg}
Assume that 
$$
v\in L^2(0,T;\cV)\quad\hbox{ and }\frac{dL_v}{dt}\in L^2(0,T;\cV')\,.
$$
Then

\smallskip
\noindent
a) the function $v$ is a.e. equal to a unique element of $C([0,T],\cH)$ still denoted $v$;

\smallskip
\noindent
b) this function $v\in C([0,T],\cH)$ satisfies
$$
\tfrac12|v(t_2)|^2_{\cH}-\tfrac12|v(t_1)|^2_{\cH}=\int_{t_1}^{t_2}\La\frac{dL_v}{dt}(t),v(t)\Ra_{\cV',\cV}dt
$$
for all $t_1,t_2\in[0,T]$
\end{Lem}

Statement a) follows from Proposition 2.1 and Theorem 3.1 in chapter 1 of \cite{LionsMage1}, and statement b) from Theorem II.5.12 of \cite{FabrieBoyer}.

\begin{Lem}\lb{L-Lempp}
Let $L\in L^2(0,T;\cV')$ satisfy
$$
\la L(t),w\ra_{\cV',\cV}=0\hbox{ for a.e. }t\in[0,T]
$$
for all $w\in\cV$. Then
$$
L(t)=0\hbox{ for a.e. }t\in[0,T]\,.
$$
\end{Lem}

\begin{proof}
Pick $\cN_w\subset[0,T]$ negligible such that $L$ is defined on $[0,T]\setminus\cN_w$ and
$$
\la L(t),w\ra_{\cV',\cV}=0\hbox{ for all }t\in[0,T]\setminus\cN_w\,.
$$
Let $\cD$ be a dense countable subset of $\cV$ and let
$$
\bar\cN:=\bigcup_{w\in\cD}\cN_w\,.
$$
For all $t\in[0,T]\setminus\bar\cN$, one has
$$
\la L(t),w\ra_{\cV',\cV}=0\hbox{ for all }w\in\cD\quad\hbox{ so that }L(t)=0
$$
because $L(t)$ is a continuous linear functional on $\cV$ and $\cD$ is dense in $\cV$.
\end{proof}

\smallskip
The next lemma recalls the functional background for Green's formula in the context of evolution equations.

\begin{Lem}\lb{L-NormTr}
Let $\Om$ be an open subset of $\bR^N$ with smooth boundary, and let $T>0$. Denote by $n$ the unit outward normal field on $\d\Om$. 
Let $\rho\in C([0,T];L^2(\Om))$ and $m\in L^2((0,T)\times\Om,\bR^N)$. Assume that 
$$
\d_t\rho+\Div_xm=0\quad\hbox{ in the sense of distributions in }(0,T)\times\Om\,.
$$
Then

\smallskip
\noindent
a) the vector field $m$ has a normal trace $m\cdot n\rstr_{(0,T)\times\d\Om}\in H^{1/2}_{00}((0,T)\times\d\Om)'$;

\smallskip
\noindent
b) for each $\psi\in H^1(\Om)$
$$
\ba
\frac{d}{dt}\int_\Om\rho(\cdot,x)\psi(x)dx&-\int_\Om m(\cdot,x)\cdot\grad_x\psi(x)dx
\\
&=-\la m\cdot n\rstr_{\d\Om},\psi\rstr_{\d\Om}\ra_{H^{-1/2}(\d\Om),H^{1/2}(\d\Om)}
\ea
$$
in $H^{-1}(0,T)$.
\end{Lem}

\begin{proof}
Let $\chi\in C^\infty_c(\bR)$ be such that
$$
\chi(t)=1\hbox{ for }t\in[-1,T+1]\quad\hbox{ and }\Supp(\chi)\subset[-2,T+2]\,.
$$
Define
$$
\bar\rho(t,x):=\left\{\ba{}&\rho(t,x)\quad&&\hbox{ if }0\le t\le T\\&\chi(t)\rho(0,x)&&\hbox{ if }t<0\\&\chi(t)\rho(T,x)&&\hbox{ if }t>T\ea\right.
$$
and
$$
\bar m(t,x):=\left\{\ba{}&m(t,x)\quad&&\hbox{ if }0\le t\le T\\&0&&\hbox{ if }t\notin[0,T]\ea\right.
$$
so that the vector field $X:=(\bar\rho,\bar m)$ is an extension of $(\rho,m)$ to $\bR\times\Om$ satisfying
$$
X\in L^2(\bR\times\Om;\bR^{N+1})\,.
$$
Besides
$$
(\d_t\bar\rho+\Div_x\bar m)(t,x)=\chi'(t)(\indc_{t<0}\rho(0,x)+\indc_{t>T}\rho(T,x))=:S(t,x)
$$
with $S\in L^2(\bR\times\Om)$ so that
$$
\Div_{t,x}X=S\in L^2(\bR\times\Om)\,.
$$
Therefore $X$ has a normal trace on the boundary $\d(\bR\times\Om)=\bR\times\d\Om$, denoted $X\cdot n\rstr_{\bR\times\d\Om}\in H^{-1/2}(\bR\times\d\Om)$.

Let $\phi\in H^{1/2}_{00}((0,T)\times\d\Om)$; denote by $\bar\phi$ its extension by $0$ to $\bR\times\d\Om$. Thus $\bar\phi\in H^{1/2}(\bR\times\d\Om)$
and there exists $\bar\Phi\in H^1(\bR\times\Om)$ such that $\bar\phi=\bar\Phi\rstr_{\bR\times\d\Om}$. The normal trace of $m$ is then defined as follows:
by Green's formula
$$
\ba
\la m\cdot n\rstr_{\bR\times\d\Om},\phi\ra_{H^{1/2}_{00}((0,T)\times\d\Om)',H^{1/2}_{00}((0,T)\times\d\Om)}
\\
:=
\la X\cdot n\rstr_{\bR\times\d\Om},\bar\phi\ra_{H^{1/2}(\bR\times\d\Om)',H^{1/2}(\bR\times\d\Om)}
\\
=
\iint_{\bR\times\Om}(\bar\rho\d_t\bar\Phi+\bar m\cdot\grad_x\bar\Phi+S\bar\Phi)(t,x)dxdt\,.
\ea
$$
Applying Green's formula on $(0,T)\times\Om$ shows that two different extensions  of the vector field $(\rho,m)$ define the same distribution 
$m\cdot n\rstr_{(0,T)\times\d\Om}$ on $(0,T)\times\d\Om$. This completes the proof of statement a). 

As for statement b), let $\ka\in H^1_0(0,T)$ and $\psi\in H^1(\Om)$, define $\Phi(t,x):=\ka(t)\psi(x)$ and let $\bar\Phi$ be the extension of $\Phi$ by $0$
to $\bR\times\Om$, so that $\bar\Phi\in H^1(\bR\times\Om)$. Thus $\phi=\Phi\rstr_{(0,T)\times\d\Om}\in H^{1/2}_{00}((0,T)\times\d\Om)$ and
$$
\ba
{}&\La\la m\cdot n\rstr_{\d\Om},\psi\rstr_{\d\Om}\ra_{H^{-1/2}(\d\Om),H^{1/2}(\d\Om)},\ka\Ra_{H^{-1}(0,T),H^1_0(0,T)}
\\
&\qquad\qquad:=
\la m\cdot n\rstr_{(0,T)\times\d\Om},\phi\ra_{H^{1/2}_{00}((0,T)\times\d\Om)',H^{1/2}_{00}((0,T)\times\d\Om)}
\\
&\qquad\qquad=
\iint_{\bR\times\Om}(\bar\rho\d_t\bar\Phi+\bar m\cdot\grad_x\bar\Phi+S\bar\Phi)(t,x)dxdt
\\
&\qquad\qquad=
\int_0^T\int_{\Om}(\rho(t,x)\ka'(t)\psi(x)+m(t,x)\cdot\grad\psi(x)\ka(t))dxdt
\\
&\qquad\qquad=
-\La\frac{d}{dt}\int_\Om\rho(t,x)\psi(x)dx,\ka\Ra_{H^{-1}(0,T),H^1_0(0,T)}
\\
&\qquad\qquad\quad\,\,+\int_0^T\int_{\Om}m(t,x)\cdot\grad\psi(x)\ka(t)dxdt
\ea
$$
which is precisely the identity in statement b).
\end{proof}

\end{appendix}
\bigskip

{\bf{Acknowledgment}}:  We thank Philippe Villedieu for suggesting the problem of deriving the temperature equations for multiphase flows in thermal 
local non-equilibrium.

\end{document}